\documentclass{amsart}

\usepackage{amssymb,epsfig,graphics}

\usepackage{color}

\newcommand{\abs}[1]{\vert#1\vert}
\newcommand{\dz}{\partial_z}

\newcommand{\curl}{\mbox{\textnormal curl }}
\newcommand{\cst}{\mbox{\textnormal{Cst}}}
\newcommand{\dt}{\partial_t}
\newcommand{\dta}{\partial_{\tau}}
\newcommand{\bU}{{\mathbf U}}

\newcommand{\bW}{{\mathbf W}}

\newcommand{\bZ}{{\mathbf Z}}
\newcommand{\eps}{\varepsilon}
\newcommand{\dsp}{\displaystyle}
\newcommand{\init}{{\vert_{t=0}}}
\newcommand{\R}{\mathbb R}
\newcommand{\T}{\mathbb T}
\newcommand{\Z}{\mathbb Z}
\newcommand{\N}{\mathbb N}
\newcommand{\C}{\mathbb C}
\newcommand{\HH}{\mathcal{H}}
\newcommand{\bk}{{\mathbf k}}
\newcommand{\cc}{\mbox{\textnormal{c.c.}}}
\newcommand{\CL}{{\mathcal C}_{{\mathcal L}}}
\newcommand{\cL}{{\mathcal L}}

\newcommand{\cLk}{{\mathcal L}(\omega,\bk)}

\newcommand{\ue}{{u_{env}}}

\newcommand{\Id}{\mbox{\textnormal{Id}}}

\newcommand{\cg}{{\bf c}_g}
\newcommand{\ttE}{{\mathtt E}}
\newcommand{\ttB}{{\mathtt B}}
\newcommand{\ttQ}{{\mathtt Q}}
\newcommand{\ttP}{{\mathtt P}}
\newcommand{\ttJ}{{\mathtt J}}
\newcommand{\tte}{{\mathtt e}}
\newcommand{\ttb}{{\mathtt b}}
\newcommand{\ttq}{{\mathtt q}}
\newcommand{\ttp}{{\mathtt p}}
\newcommand{\fre}{{\mathfrak e}}

\newcommand{\tr}{\tilde{\rho}}
\newcommand{\at}{\makeatletter @\makeatother}
\newtheorem{assumption}{Assumption}[section]
\newtheorem{theorem}{Theorem}[section]
\newtheorem{openproblem}{Open Problem}
\newtheorem*{acknowledgements}{Acknowledgements}

\newtheorem{coro}{Corollary}[section]
\theoremstyle{remark}
\newtheorem{remark}{Remark}[section]
\newtheorem{example}{Example}[section]
\newtheorem{notation}{Notation}[section]

\title[Variant of the focusing NLS equation]{Variants of the
  focusing NLS equation. Derivation, justification and open problems
  related to filamentation}

\thanks{D. L. acknowledges support from the ANR-13-BS01-0003-01
  DYFICOLTI, the ANR BOND.}

\author{\'Eric Dumas, David Lannes and J\'er\'emie Szeftel}

\address{{\rm \'Eric Dumas:} Institut Fourier - 
Universit\'e Grenoble 1 - 100, rue des math\'ematiques - 
BP 74 - 38402 Saint Martin d'H\`eres - France. {\it Email address: }
{\rm Eric.Dumas\at ujf-grenoble.fr}}

\address{{\rm David Lannes:} DMA - \'Ecole Normale Sup\'erieure - 
45, rue d'Ulm - 75005 Paris - France. {\it Email address: } 
{\rm David.Lannes\at ens.fr}}

\address{{\rm J\'er\'emie Szeftel:} Laboratoire Jacques-Louis Lions - 
Universit\'e Pierre et Marie Curie - 
4, place Jussieu - 75005 Paris - France. 
{\it Email address: }{\rm Jeremie.Szeftel\at upmc.fr}}

\begin{document}

\begin{abstract}
The focusing cubic NLS is a canonical model for the propagation 
of laser beams. In dimensions 2 and 3, it is known that a 
large class of initial data leads to finite time blow-up. 
Now, physical experiments suggest that this blow-up does not 
always occur.  This might be explained by the fact that 
some physical phenomena neglected by the standard
NLS model become relevant at large intensities of the beam. 
Many ad hoc variants of the focusing NLS equation have been 
proposed to capture such effects. 
In this paper, we derive some of these variants 
from Maxwell's equations and propose some new ones. 
We also provide rigorous error estimates for all the models 
considered. 
Finally, we discuss some open problems related to these 
modified NLS equations. 
\end{abstract}

\maketitle

\section{Introduction}

The cubic, focusing, nonlinear Schr\"odinger equation 
in space dimension $d$ is given by
\begin{equation}\label{eqNLS}
\left\{\begin{array}{l} i\dta v+\Delta v+\abs{v}^2 v=0, \quad 
\tau> 0, \quad x\in \R^d, \vspace{1mm} \\
v(0,x)=v_0(x),\,x\in \R^d.
\end{array}\right.
\end{equation}
It is a canonical model for the propagation of laser beams.

 From a result of Ginibre and Velo \cite{GV}, equation (\ref{eqNLS}) is locally well-posed in $H^1=H^1(\R^d)$ for $d=1, 2, 3$, and thus, for $v_0\in H^1$, there exists $0<T\leq +\infty$ and a unique solution $v(\tau)\in {\mathcal{C}}([0,T),H^1)$ to (\ref{eqNLS}) and either $T=+\infty$, we say the solution is global, or $T<+\infty$ and then $\lim_{\tau\uparrow T}\|\nabla v(t)\|_{L^2}=+\infty$, we say the solution blows up in finite time. 

The NLS equation (\ref{eqNLS}) also admits the following (formal) 
conservation laws: 
$$
\begin{array}{lll}
L^2-\mbox{norm}: \ \ \|v(\tau)\|^2_{L^2}=\|v_0\|_{L^2}^2; \vspace{1mm} \\
\displaystyle 
\mbox{Energy}:\ \ E(v(\tau))=\frac{1}{2}\int|\nabla v(\tau,x)|^2{\rm d}x
-\frac{1}{4}\int |v(\tau,x)|^{4}{\rm d}x=E(v_0); \vspace{1mm} \\
\displaystyle 
\mbox{Momentum}:\ \ Im \left( \int\nabla v(\tau,x)\overline{v(\tau,x)}{\rm d}x \right) =
Im \left( \int\nabla v_0(x)\overline{v_0(x)}{\rm d}x \right).
\end{array}
$$
It is also known that a large group of symmetries leaves the equation invariant: if $v(\tau,x)$ solves \eqref{eqNLS}, then $\forall (\lambda_0,\tau_0,x_0,\beta_0,\gamma_0)\in \R_*^+\times\R\times\R^d\times\R^d\times \R$, so does 
\begin{equation}
\label{symmetrygroup}
u(\tau,x)=\lambda_0v(\lambda_0^2\tau+\tau_0,\lambda_0 x+x_0-\beta_0t)e^{i\frac{\beta_0}{2}\cdot(x-\frac{\beta_0}{2} \tau)}e^{i\gamma_0}.
\end{equation}
The scaling symmetry $u(\tau,x)=\lambda_0v(\lambda^2_0\tau,\lambda_0 x)$ leaves the homogeneous Sobolev space $\dot{H}^{s_c}(\R^d)$ invariant, where $s_c=\frac{d}{2}-1$. 

Referring to conservation of the $L^2$ norm by the flow, 
\eqref{eqNLS} is said to be $L^2-$subcritical if $s_c<0$, 
$L^2-$critical if $s_c=0$ and $L^2-$supercritical if $s_c>0$. 
Thus, \eqref{eqNLS} is $L^2-$subcritical if $d=1$, $L^2-$critical 
if $d=2$, and $L^2-$supercritical if $d\geq 3$. 
In the subcritical case, global existence 
(in ${\mathcal{C}}([0,\infty),L^2)$) holds for arbitrarily 
large data in $L^2$. It turns out that in this case, 
global existence (in ${\mathcal{C}}([0,\infty),H^1)$) 
also holds for arbitrarily large data in $H^1$, 
due to the conservation of mass and energy. 
In the critical and supercritical cases however, there exist stable finite time blow-up dynamics. This has been known since the 60ies using global obstructive arguments based on the virial identity (see e.g. \cite{ZS}).\footnote{Much more is known about the finite time blow up dynamics for the focusing NLS and we refer the interested reader to \cite{SS} \cite{MR1} \cite{MR2} \cite{MR3} \cite{MR4} \cite{MR5} \cite{MR6} \cite{MRS} and references therein.}

There is however a discrepancy between the blow-up results predicted by (\ref{eqNLS})
and physical observations. Indeed, while the blow-up signifies a break-down of the solution $v$,
physical observations show in many cases that lasers begin to focus according to the scenarios
associated to (\ref{eqNLS}) but depart from this behavior slightly
before the focusing time. The reason advanced by physicists is that
some physical phenomena that have been neglected to derive
(\ref{eqNLS}) become relevant at high intensities, and therefore near
focusing. This phenomenon is called {\it filamentation}: defocusing
physical phenomena are triggered at high intensities and halt the
collapse of the beam. This interplay between diffraction,
self-focusing, and defocusing mechanisms allow for the beam to
propagate along several times the focusing distance (called Rayleigh
length in optics) and the resulting
structure is called {\it filament} .

Many variants of (\ref{eqNLS}) have been derived in optics  to take
into account these additional physical phenomena and reproduce the
filamentation mechanism. In many cases, it is a mathematical open problem to prove whether 
these additional terms prevent focusing or not, and a fortiori to understand the modification of the dynamics induced by them.

Rather than adding as usual ad hoc modifications to (\ref{eqNLS}) in order to take new physical effects into
account, we choose here to rigorously derive such modifications from
Maxwell's equations.  We then comment on some of the most physically
relevant open mathematical problems that these modified equations
raise and that are natural milestones towards the understanding of filamentation. These variants can roughly be classified into two groups, depending on whether they 
 take ionization processes into account or not. 

\begin{notation}
In the brief presentation below,  we denote by $z$ the direction of propagation of the laser,
by $X_\perp=(x,y)$ the transverse variables, and by $\Delta_\perp=\partial_x^2+\partial_y^2$ the transverse Laplace operator.
In dimension $d=2$, the variable $y$ is omitted (and hence $\Delta_\perp=\partial_x^2$), while in dimension $d=1$, 
functions depend only on $z$ (so that $\Delta_\perp=0$).
\end{notation}

\medbreak

\noindent
\textsc{Models without ionization processes}\\
 We give below a family of variants to (\ref{eqNLS})
that incorporate many physical phenomena neglected by (\ref{eqNLS}). It is of course possible to look at one 
or several of these additional effects simultaneously. 
We state the equations in their most general form, starting 
with a family of \emph{scalar} NLS equations, 
and then give the corresponding \emph{vectorial} 
-- and more general -- form of these equations. 
Let us therefore consider
\begin{equation}\label{NLSfamily1}
i P_2(\eps \nabla)\dta v + (\Delta_\perp +\alpha_1\dz^2) v +i \alpha_2 v +(1+i \eps \boldsymbol{\alpha}_3\cdot \nabla )
\big[\big(1+f(\eps^{r} \abs{v}^2)\big)\abs{v}^2 v \big]=0,
\end{equation}
where $v$ is a complex-valued function. Here, $\eps>0$ 
is a (small) parameter; $P_2(\eps\nabla)$ is a second (at most) 
order, self-adjoint, positive operator; $\alpha_1=0,\pm 1$; 
$\alpha_2\geq 0$; $\boldsymbol{\alpha}_3\in \R^d$; 
$f:\R^+\to \R$ is a smooth mapping vanishing at the origin, 
and $r > 0$. 
The physical meaning of these terms is commented below:
\begin{enumerate}
\item Nonlinearity. The cubic nonlinearity in (NLS) corresponds to a first order approximation of the nonlinear optical phenomena. At high intensities, it is often worth including some next order terms captured here
by the additional term $f(\eps^r\abs{v}^2)$. We consider here three situations:
\begin{enumerate}
\item Cubic nonlinearity: $f=0$. 
\item Cubic/quintic nonlinearity: $f(r)=-r$.
\item Saturated nonlinearity: $f$ is a smooth function on $\R^+$ vanishing at the origin and such that
$(1+f(r))r$ is bounded on $\R^+$ (e.g. $f(r)=-\frac{r}{1+r}$).
\end{enumerate}
\item Group velocity dispersion (GVD). The coefficient $\alpha_1$ accounts for the dispersion of the group velocity and three different situations are possible:
\begin{enumerate}
\item No GVD:  $\alpha_1=0$.
\item Anomalous GVD: $\alpha_1=1$.
\item Normal GVD: $\alpha_1=-1$. 
\end{enumerate}
\item Damping. The coefficient $\alpha_2$ accounts for damping phenomena:
\begin{enumerate}
\item No damping: $\alpha_2=0$.
\item Damping: $\alpha_2>0$.
\end{enumerate}
\item\label{offaxis} Off-axis variations of the group velocity. The operator $P_2(\eps\nabla)$ is here to account for the fact that self-focusing pulses become asymmetric due to the variation of the group velocity of off-axis rays\footnote{This phenomenon is often referred to in optics as \emph{space-time focusing} \cite{Rothenberg}; we do not use this terminology here because this would be misleading. Indeed, physicists usually take $z$ as the evolution variable and treat $t$ as a space variable. This amounts to permuting $t$ and $z$ in (\ref{NLSfamily1}) and elsewhere. }. The operator $P_2(\eps \nabla)$ is a second order, self-adjoint, and positive operator in the sense that
$(P_2(\eps\nabla) u,u)\geq C \abs{u}^2_*$, 
with $\abs{u}^2_*\geq \abs{u}_2^2$. The norm $\abs{\cdot}_*$ 
may also control derivatives of $u$; we consider three cases:
\begin{enumerate}
\item No off-axis dependence: $P_2(\eps \nabla)=1$, 
and therefore $\abs{\cdot}_*=\abs{\cdot}_2$.
\item Full off-axis dependence: the norm $\abs{\cdot}_*$ 
controls all first order derivatives, $\abs{u}_*^2 \sim \abs{u}_2^2+\eps^2 \abs{\nabla u}_2^2$.
\item Partial off-axis dependence: the norm $\abs{\cdot}_*$ controls some 
but not all first order derivatives. More precisely, there exists $j$ ($j<d$) 
linearly independent vectors 
${\bf v}_j\in \R^d$ such that
$\abs{u}_*^2 \sim \abs{u}_2^2+\eps^2 \sum_{k=1}^j 
\abs{{\bf v}_k\cdot \nabla u}_2^2$.
\end{enumerate}
\item Self-steepening of the pulse. The operator $(1+i\eps \boldsymbol{\alpha}_3\cdot \nabla)$ in front of the nonlinearity
accounts for off-axis dependence of the nonlinearity, responsible for the possible formation of \emph{optical shocks}. Various cases are considered here:
\begin{enumerate}
\item No self-steepening. This corresponds to $ \boldsymbol{\alpha}_3=0$ and to the usual situation where the nonlinearity
does not contain any derivative.
\item Longitudinal self-steepening. When $ \boldsymbol{\alpha}_3$ is colinear to ${\bf e}_z$, there is a derivative in the nonlinearity along the direction $z$ of propagation of the laser.
\item Transverse self-steepening. When $ \boldsymbol{\alpha}_3\neq 0$ and $ \boldsymbol{\alpha}_3\cdot {\bf e}_z=0$, there is a derivative in the nonlinearity along a direction orthogonal to the direction of propagation.
\item Oblique self-steepening. When $ \boldsymbol{\alpha}_3$ is neither colinear nor orthogonal to ${\bf e}_z$.
\end{enumerate}
\end{enumerate}

\begin{remark}
The standard Schr\"odinger equation (\ref{eqNLS}) is obtained with $P_2(\eps\nabla)=1$, $\alpha_2=0$, $ \boldsymbol{\alpha}_3=0$, $f=0$, and $\alpha_1=1$. Using the above terminology, it corresponds to a cubic nonlinearity, without damping terms, off-axis variation of the group velocity and self-steepening, and with anomalous GVD.
\end{remark}

As previously said, (\ref{NLSfamily1}) stems from a more general \emph{vectorial} equation. 
For the sake of simplicity, we
give here the equation corresponding to the cubic case (or $f=0$ in (\ref{NLSfamily1})):
$$\eqref{NLSfamily1}_{\rm vect}\qquad
i P_2(\eps \nabla)\dta {\bf v} + (\Delta_\perp +\alpha_1\dz^2) {\bf v} +i \alpha_2 {\bf v}
 +\frac{1}{3}(1+i \eps  \boldsymbol{\alpha}_3\cdot \nabla )
\big[({\bf v}\cdot {\bf v})\overline{{\bf v}}+2\abs{{\bf v}}^2 {\bf v} \big]=0,
$$
where ${\bf v}$ is now a $\C^2$-valued function.

\begin{remark}\label{remvect}
Equation (\ref{NLSfamily1}) is in fact a particular case of (\ref{NLSfamily1})$_{\rm vect}$ corresponding to
initial data living on a one-dimensional subspace of $\C^2$. Indeed, if the initial condition to  (\ref{NLSfamily1})$_{\rm vect}$ has the form ${\bf v}_{\vert_{\tau=0}}=v^0(x){\bf v}_0$ with ${\bf v}_0\in \R^2$ and $v^0$ a scalar-valued function, then the solution to  (\ref{NLSfamily1})$_{\rm vect}$ takes the form ${\bf v}(\tau,x)=v(\tau,x){\bf v}_0$, where $v$ solves (\ref{NLSfamily1}) with initial condition $v^0$.
\end{remark}

\noindent
\textsc{Models with ionization processes}\\
In addition to the physical phenomena taken into account in (\ref{NLSfamily1}), 
it is necessary at high intensities to include ionization processes for a correct description of the laser pulse. The reason why this phenomenon is singled out here is because a system of two equations must be considered instead of the single equation (\ref{NLSfamily1}). In the most simple case (i.e. $P_2=1$, $f=0$, $\alpha_2=0$,
$ \boldsymbol{\alpha}_3=0$; for a more general model, see \eqref{NLSfamily2gen}), this system is given by
\begin{equation}\label{NLSfamily2}
\left\lbrace
\begin{array}{l}
\!\!\!\dsp i (\dt + c_g\dz ) u
+\eps (\Delta_\perp \!+\!\alpha_1\dz^2) u 
+ \eps ( \abs{u}^2 \! - \! \rho ) u  =
- i \eps c(\alpha_4 \abs{u}^{2K-2}\! u+\alpha_5\rho u) , \vspace{1mm} \\
\!\!\!\dsp \dt \rho = \eps \alpha_4 \abs{u}^{2K}+\eps \alpha_5 \rho \abs{u}^2,
\end{array}\right.
\end{equation}
with $\alpha_4,\alpha_5\geq 0$, $c>0$, and where $\rho$ is the density
of electrons created by ionization, while $\cg=c_g {\bf e}_z$ is the
group velocity associated to the laser pulse.

The system \eqref{NLSfamily2} does not directly compare to
\eqref{NLSfamily1} and \eqref{NLSfamily1}$_{\rm vect}$; indeed,
\eqref{NLSfamily2} is written in the fixed frame of the laboratory,
while \eqref{NLSfamily1} and \eqref{NLSfamily1}$_{\rm vect}$ are written
in a frame moving at the group velocity $\cg=c_g{\bf e}_z$ and with
respect to a rescaled time $\tau=\eps t$. Rather than
\eqref{NLSfamily2}, the NLS equation with ionization used in the
physics community is  its version written in the same variables
as \eqref{NLSfamily1} and \eqref{NLSfamily1}$_{\rm vect}$. More precisely,
if we set 
$$
u(t,X_\perp,z) =v(\eps t,X_\perp,z-c_g t),\qquad \rho(t,X_\perp,z) =\tilde \rho(\eps t,
X_\perp,z-c_g t),
$$
and $\tau=\eps t$, the equations \eqref{NLSfamily2} are {\it approximated}
\footnote{The approximation lies in the equation on 
$\rho$. In the new variables, the second equation of 
\eqref{NLSfamily2} is given by 
$$
\eps \dta \tilde\rho -c_g\dz \tilde\rho=\eps \alpha_4 \abs{v}^{2K}+\eps \alpha_5 \tilde\rho \abs{v}^2.
$$
In the physics literature, the term $\eps \dta \tilde\rho$ is neglected, and
this corresponds to \eqref{NLSfamily3}.
} by the following ones,
\begin{equation}\label{NLSfamily3}
\left\lbrace
\begin{array}{l}
\dsp i \dta v
+(\Delta_\perp +\alpha_1\dz^2) v +  (\abs{v}^2 -\tilde\rho) v  = 
- i c(\alpha_4 \abs{v}^{2K-2}v+\alpha_5\tilde\rho v) , \vspace{1mm} \\
\dsp -c_g\dz \tilde\rho = \eps \alpha_4 \abs{v}^{2K}+\eps \alpha_5 \tilde\rho \abs{v}^2.
\end{array}\right.
\end{equation}

The ionization processes taken into account by the systems
\eqref{NLSfamily2} and \eqref{NLSfamily3} are:
\begin{enumerate}
\item Photo-ionization. This corresponds to $\alpha_4>0$ and $K>0$ ($K$ is the number of photons necessary to liberate one electron).
\item Collisional ionization. When $\alpha_5>0$, a term corresponding to collisional ionization is added to the evolution equation on $\rho$.
\end{enumerate}
\begin{remark}
\item[(1)] The coupling with $\rho$ can of course be added to 
any equation of the family  (\ref{NLSfamily1}). 
\item[(2)] A vectorial variant of  (\ref{NLSfamily2}) and \eqref{NLSfamily3} can also 
be derived in the same lines as (\ref{NLSfamily1})$_{\rm vect}$.
\item[(2)] When $\alpha_4=\alpha_5=0$ and $\rho_{\vert_{t=0}}=0$
  (respectively $\lim_{z\to -\infty} \tilde\rho=0$)
one recovers  (\ref{NLSfamily1}) from  (\ref{NLSfamily2})
(respectively \eqref{NLSfamily3}).
\end{remark}

The rest of the paper is as follows. In section \ref{sectMab}, 
we recall the Maxwell equations, we give an abstract formulation 
and we discuss the spaces of initial data used for the Cauchy 
problem. 
In section \ref{sectDer}, we prove our main result about the 
rigorous derivation of general abstract versions of the models 
\eqref{NLSfamily1}, \eqref{NLSfamily1}$_{\rm vect}$, and 
\eqref{NLSfamily2}, \eqref{NLSfamily3}. In section \ref{sec:openproblem}, we analyze 
the role of the various parameters in \eqref{NLSfamily1} and 
\eqref{NLSfamily2} or \eqref{NLSfamily3} In particular, we consider whether they 
indeed prevent the breakdown in finite time or not. We also 
formulate a number of interesting open problems for these 
modified NLS equations. Finally, an appendix contains explicit 
computations for a physically relevant system of Maxwell 
equations which show that the abstract models derived in Section 
\ref{sectDer} take indeed the form of \eqref{NLSfamily1}, 
\eqref{NLSfamily1}$_{\rm vect}$, and \eqref{NLSfamily2}.

\subsection{Notations} We denote by \\
- ${\mathcal F}u(\xi)=\widehat{u}(\xi)$, the Fourier transform of $u$ with respect to the space variables $x\in \R^d$. \\
- ${\mathcal F}_tu(\omega)$, the Fourier transform of $u$ with respect to the time variable $t$.\\
- $f(D)$, the Fourier multiplier with symbol $f(\xi)$: $\widehat{f(D)u}(\xi)=f(\xi)\widehat{u}(\xi)$.\\
- $f(D_t)$, Fourier multipliers with respect to time.\\
- $\Lambda=(1-\Delta)^{1/2}$, the Fourier multiplier with symbol $(1+\abs{\xi})^{1/2}$.

\section{The Maxwell equations and an abstract mathematical formulation}\label{sectMab}

\subsection{The Maxwell equations}

The Maxwell equations in a non magnetizable medium are a set of two equations coupling the evolution of the \emph{magnetic field} $\ttB$ to the \emph{electric induction} ${\mathtt D}$,
\begin{equation}\label{Max1}
\left\lbrace
\begin{array}{l}
\dsp \dt \ttB+\curl \ttE=0,\\
\dsp \dt {\mathtt D}-\frac{1}{\mu_0}\curl \ttB= -\ttJ,
\end{array}\right.
\end{equation}
where  ${\mathtt D}$ is given in terms of the \emph{electric field} $\ttE$ and a \emph{polarization} $\ttP$ --- modeling the way
the dipole moment per unit volume depends on the strength of the electric field --- by
the relation
\begin{equation}\label{eqDB}
{\mathtt D}=\epsilon_0 \ttE+\ttP,
\end{equation}
and where we used standard notations $\epsilon_0$ and $\mu_0$ for the \emph{electric permittivity} and \emph{magnetic permeability} in vacuum.
The evolution equations (\ref{Max1}) go along with two constitutive laws, 
\begin{equation}\label{eqdiv}
\nabla\cdot {\mathtt D}=\rho,\qquad \nabla \cdot \ttB=0,
\end{equation}
where $\rho$ is the electric charge density.\\
As a consequence of the relation $\nabla\cdot {\mathtt D}=\rho$ and the second equation of (\ref{Max1}) we get the continuity equation
coupling $\rho$ to the \emph{current density} $\ttJ$,
\begin{equation}\label{trucmuche}
\dt \rho+\nabla\cdot \ttJ=0.
\end{equation}
Introducing the speed of light in vacuum
$$
c=\frac{1}{\sqrt{\epsilon_0\mu_0}},
$$
the equations (\ref{Max1}) can also  be rewritten as a set of two evolution equations on the magnetic field $\ttB$
and the electric field $\ttE$,
\begin{equation}\label{Max2}
\left\lbrace
\begin{array}{l}
\dsp \dt \ttB+\curl \ttE=0,\\
\dsp \dt \ttE- c^2 \curl \ttB=-\frac{1}{\epsilon_0}\dt \ttP -\frac{1}{\epsilon_0} \ttJ.
\end{array}\right.
\end{equation}
In order to get a closed system of equations, we still need two physical informations:
\begin{enumerate}
\item A description of the  polarization response to the electric field $\ttE$.
\item A description of the current density $\ttJ$.
\end{enumerate}
We first address the description of the polarization response in absence of current density and then proceed to describe the 
modification to be made when current density is included.

\subsubsection{The polarization response to the electric field}

Throughout this section, we assume that there is no charge nor current density ($\rho=0$, $\ttJ=0$). The general case will be handled in \S \ref{sectgeneral} below.

There exist various ways to describe the polarization $\ttP$; we use here a simple and natural model called ``nonlinear anharmonic oscillator'', according to which the polarization is found by solving the second order ODE
\begin{equation}\label{eqpolarisation}
\dt^2 \ttP+\Omega_1\dt \ttP+\Omega_0^2 \ttP-\nabla V_{NL}(\ttP)=\epsilon_0 b \ttE,
\end{equation}
where $b\in\R$ is a coupling constant and $\Omega_0,\Omega_1>0$ 
are frequencies,  and where $V_{NL}$ accounts for nonlinear effects. When such effects are
neglected, the description (\ref{eqpolarisation}) goes back to Lorentz \cite{Lorentz} and expresses the fact that electrons are bound to
the nucleus by a spring force. Nonlinearities have been added to this description by Bloembergen \cite{Bloembergen} and Owyyoung \cite{Owyoung} and the mathematical investigation of their influence was initiated by Donnat, Joly, M\'etivier and Rauch \cite{DJMR,JMR} (see also \cite{Lannes}).
\begin{remark}
In physics books, the polarization $\ttP$ is often sought as an expansion
$$
\ttP=\epsilon_0\big[\chi^1 [\ttE]+\chi^2 [\ttE,\ttE]+\chi^3 [\ttE,\ttE,\ttE]+\dots \big],
$$
where the operator $\chi^1$ is  called  the \emph{linear susceptibility} of the material, while for $j>1$, the operators
 $\chi^j$ are the \emph{$j$-th order nonlinear susceptibilities}. It is easy to check that the linear susceptibility
corresponding to (\ref{eqpolarisation}) is given by the nonlocal (in time) operator
$$
\chi^1[\ttE]=\chi^1(D_t)\ttE\quad\mbox{ with }\quad
\chi^1(\omega)=\frac{b}{\Omega_0^2-\omega^2+i\Omega_1\omega},
$$ 
where we used the Fourier multiplier notation,
$$
{\mathcal F}_t [\chi^1(D_t)\ttE](\omega)=\chi^1(\omega){\mathcal F}_t \ttE(\omega).
$$
\end{remark}
\begin{example}\label{exNL}
Typical examples for $V_{NL}(\ttP)$ are
\begin{itemize}
\item[(i)] Cubic nonlinearity: 
$$
V_{NL}(\ttP)=\frac{a_3}{4} \abs{\ttP}^4\quad\mbox{ and therefore } \quad \nabla V_{NL}(\ttP)=a_3\abs{\ttP}^2 \ttP 
$$
\item[(ii)] Cubic/quintic  nonlinearity:
$$
V_{NL}(\ttP)=\frac{a_3}{4} \abs{\ttP}^4-\frac{a_5}{6}\abs{\ttP}^6\quad\mbox{ and therefore } \quad \nabla V_{NL}(\ttP)=a_3\abs{\ttP}^2\ttP -a_5\abs{\ttP}^4\ttP.
$$
\item[(iii)] Saturated nonlinearity: there exists a function $v_{sat}: \R^+\to \R$, with $v_{sat}'$ and $v_{sat}''$ bounded on
$\R^+$ and such that
$$
V_{NL}(\ttP)=\frac{1}{2}v_{sat}(\abs{\ttP}^2) \quad\mbox{ and therefore } \quad \nabla V_{NL}(\ttP)=v_{sat}'(\abs{\ttP}^2)\ttP;
$$
for instance, one can take 
$$
v_{sat}(r)=\frac{a_3}{2}\frac{r^2}{1+ \frac{2a_5}{3a_3}r},
$$
in which case  $\nabla V_{NL}(\ttP)=a_3\abs{\ttP}^2\ttP -a_5\abs{\ttP}^4\ttP+h.o.t$, and is therefore the same
at the origin as in (ii) above, up to higher order terms (seventh order terms here).
\end{itemize}
\end{example}
We show in Appendix \ref{NDwithout} that Maxwell's equations can be put under the following dimensionless form\footnote{The constitutive laws (\ref{eqdiv}) are omitted because they are propagated by the equations if they are initially satisfied.} for all
the nonlinearities considered in Example \ref{exNL},
\begin{equation}\label{Max3}
\left\lbrace
\begin{array}{l}
\dsp \dt \ttB+ \curl \ttE=0,\vspace{1mm}\\
\dsp \dt \ttE- \curl \ttB+\frac{1}{\eps}\sqrt{\gamma} \ttQ^\sharp=0,\vspace{1mm}\\ 
\dsp \dt \ttQ^\sharp+\eps^{1+p}\omega_1 \ttQ^\sharp-\frac{1}{\eps}(\sqrt{\gamma} \ttE-\omega_0 \ttP^\sharp)=\eps \frac{\gamma}{\omega_0^3}
\big(1+f(\eps^r\abs{\ttP^\sharp}^2)\big)\abs{ \ttP^\sharp}^2\ttP^\sharp,
\vspace{1mm}\\ 
\dsp \dt \ttP^\sharp-\frac{1}{\eps}\omega_0 \ttQ^\sharp=0,
\end{array}\right.
\end{equation}
where $\gamma$, $\omega_0$, $\omega_1$, $r$ and $p$ are constants, $0<\eps\ll 1$ is a small parameter (the ratio of the duration of an optical cycle over the duration of the pulse, see Appendix \ref{NDwithout}), while $f$ is a smooth function vanishing at the origin.

\subsubsection{The case with charge and current density} \label{sectgeneral}

The main mechanism at stake in laser filamentation is certainly the
local ionization of the medium: once a powerful self-focusing laser
beam reaches high enough intensities, it ionizes the medium around itself. It leaves behind a narrow
channel of plasma, hereby  causing local defocusing that prevents blowup.\\
Taking current density into account, we come back to the set 
of equations \eqref{Max2}-\eqref{eqpolarisation}, and 
a physical description of the current density $\ttJ$ is needed.
This current density has the form
\begin{equation}\label{currenttotal}
\ttJ=\ttJ_{\rm e}+\ttJ_{\rm i},
\end{equation}
where $\ttJ_{\rm e}$ and $\ttJ_{\rm i}$ are respectively the free electron and ionization 
current densities. 

\noindent
- {\it Free electron current density}. Partial ionization of the material medium by the laser 
generates free electrons, with charge 
$q_{\rm e}(=-1.6\times10^{-19}C)$. This induces 
a free electron current density $\ttJ_{\rm e} = q_{\rm e} \rho_{\rm e} v_{\rm e}$,
where $\rho_{\rm e}$ is the electron density, and $v_{\rm e}$ is the 
electron velocity. A rough\footnote{This approximation can be
  deduced formally by assuming that ions are at rest and that
  electron motion is described by the compressible Euler 
system (see for instance \cite{BergeSkupin}).
Neglecting electron collisions, such a
  model yields $\dt \ttJ_{\rm e}=\frac{q_{\rm e}^2}{m_{\rm
      e}}\rho_{\rm e}\ttE$ which formally yields \eqref{modBS}
  assuming that $\ttE$ is as in \eqref{modEF} and that $\rho_{\rm e}$
  is not oscillating at leading order. It would of course be interesting
to provide a rigorous justification to these approximations.}, but standard model in nonlinear
optics is to take (see \cite{BergeSkupin2} and references therein),
\begin{equation}\label{modEF}
E(t,X)\sim E_{01}(t,X)e^{i({\bf k}_{\rm l}\cdot X-\omega_{\rm l} t)}+\cc,
\end{equation}
where ${\bf k}_{\rm l}$ and $\omega_{\rm l}$ are the laser wave number and
pulsation respectively, with $\abs{\dt E_{01}}\ll \abs{\omega_{\rm l}
  E_{01}}$ (slowly varying envelope approximation); the polarization current is then
taken under the form
\begin{equation}\label{modBS}
\ttJ_{\rm e}\sim J_{01}(t,X)e^{i({\bf k}_{\rm l}\cdot X-\omega_{\rm l} t)}+\cc
\quad\mbox{ with }\quad
J_{01}=i \frac{q_{\rm e}^2}{\omega_{\rm l} m_{\rm e}}\rho_{\rm e} E_{01},
\end{equation}
where $m_{\rm e}$ is the electron mass.
The drawback of this model is that it assumes that the electric field
and the current density field can be written at leading order as wave
packets (i.e. are given under the form \eqref{modEF}-\eqref{modBS}). In
particular, it does not provide any relation between the current
density $\ttJ$ and the electric field $\ttE$ that could be used in
Maxwell's equations \eqref{Max2}. We therefore propose here such a
relation, namely,
$$
\ttJ_{\rm e}=
\frac{q_{\rm e}^2}{\omega_{\rm l} m_{\rm e}}
{\mathcal H}\left(\frac{{\bf k}_{\rm l}}{k_{\rm l}^2}\cdot D\right)(\rho_{\rm e}\ttE),
$$
where $k_{\rm l}=\abs{{\bf k}_{\rm l}}$, and
${\mathcal H}$ is the regularization of the Hilbert
transform given by the Fourier multiplier 
\begin{equation}\label{Hilbert}
{\mathcal H}(D_z)=\frac{\sqrt{2}~i D_z}{(1+D_z^2)^{1/2}}.
\end{equation}
Quite obviously, this is consistent with the usual model \eqref{modBS}
since this latter is recovered at leading
order when the electric field is a wave packet under the form \eqref{modEF}.

Finally, the evolution of the electron density $\rho_{\rm e}$ is given by a source 
term $S$ representing external plasma sources. 
Taking into account photo-ionization and collisional 
ionization, but neglecting electron recombination
(see for instance \cite{BergeSkupin} for richer models), 
we have
$$
S = W(I)(\rho_{\rm nt}-\rho_{\rm e}) 
+ \frac{\sigma}{U_{\rm i}} \rho_{\rm e} I, 
$$
where the intensity is $I=|E|^2$ and $\rho_{\rm nt}$ is the 
constant density of neutral species. In the regime considered 
here\footnote{For higher intensities, electrons can 
tunnel out the Coulomb barrier of atoms, and $W(I)$ is 
modified.}, $\rho_{\rm e}$ is negligible compared to $\rho_{\rm nt}$ 
and the photo-ionization rate $W(I)$ takes the form 
$$
W(I) = \sigma_K I^{2K},
$$ 
for some constant coefficient $\sigma_K>0$ and with 
$K>1$ the number of photons needed to liberate 
one electron. The collisional ionization cross-section 
$\sigma$ depends on the laser frequency, 
and $U_{\rm i}$ is the ionization potential. 
Summing up, we get the following expression for the free electron current 
$\ttJ_{\rm e}$ and $\rho=\rho_{\rm e}$,
\begin{equation}\label{eqrhoJ}
\left\lbrace\begin{array}{l}
\dsp \ttJ_{\rm e}=\frac{q_{\rm e}^2}{\omega_{\rm l} m_{\rm
    e}}{\mathcal H}\left(\frac{{\bf k}_{\rm l}}{k^2_{\rm l}}\cdot D\right)(\rho\ttE),\\
\dsp \dt \rho = \sigma_K \rho_{\rm nt} \abs{\ttE}^{2K}
+ \frac{\sigma}{U_{\rm i}}\rho\abs{\ttE}^2.
\end{array}\right.
\end{equation}
- {\it Ionization current density}. It is also necessary to take into account losses due to
photo-ionization. We therefore introduce a ionization current density
$\ttJ_{\rm i}$ such that $\ttJ_{\rm i}\cdot E$ represents the energy lost by the laser
to extract electrons (per time and volume unit). More precisely, 
$\ttJ_{\rm i}\cdot E$ is equal to the energy necessary to
extract one electron (given by the ionization potential $U_{\rm i}$) multiplied by the
number of electrons per time and volume unit (given by $\dt
\rho$). Using the second equation of (\ref{eqrhoJ}), this gives
$$
\ttJ_{\rm i}\cdot \ttE = 
U_{\rm i}\sigma_K\rho_{\rm nt}\abs{\ttE}^{2K}+\sigma\rho \abs{\ttE}^2.
$$
We therefore take
\begin{equation}\label{currention}
\ttJ_{\rm i}=\big(U_{\rm i}\sigma_K\rho_{\rm nt}\abs{\ttE}^{2K-2}+\sigma\rho)\ttE.
\end{equation}

We show in Appendix \ref{NDwith} that after 
nondimensionalization, the set of equations 
\eqref{Max2}-\eqref{currenttotal}-\eqref{eqpolarisation}-\eqref{eqrhoJ}-\eqref{currention} 
(for the nonlinearities considered in Example \ref{exNL})
becomes, 
\begin{equation}\label{Max4}
\left\lbrace
\begin{array}{l}
\dsp \dt \ttB + \curl \ttE = 0 , \vspace{1mm} \\
\dsp \dt \ttE - \curl \ttB + \frac{1}{\eps}\sqrt{\gamma} \ttQ^\sharp =
-\eps  {\mathcal H}\big(\eps\frac{{\bf k}}{k^2}\cdot D_x\big)(\rho \ttE)
-\eps c_0 \big(c_1
\abs{\ttE}^{2K-2}+ c_2\rho\big)\ttE  , \vspace{1mm} \\ 
\dsp \dt \ttQ^\sharp 
+ \eps^{1+p}\omega_1 \ttQ^\sharp 
- \frac{1}{\eps}(\sqrt{\gamma} \ttE 
- \omega_0 \ttP^\sharp)=\eps \frac{\gamma}{\omega_0^3}
\big(1+f(\eps^r\abs{\ttP^\sharp}^2)\big)
\abs{ \ttP^\sharp}^2\ttP^\sharp , \vspace{1mm} \\ 
\dsp \dt \ttP^\sharp 
- \frac{1}{\eps}\omega_0 \ttQ^\sharp=0 , \vspace{1mm} \\ 
\dsp \dt \rho = 
\eps c_1 \abs{\ttE}^{2K} + \eps c_2 \rho \abs{\ttE}^2, 
\end{array}\right.
\end{equation}
with the same notations as in \eqref{Max3} for the constants 
$\gamma$, $\omega_0$, $\omega_1$, $r$ and $p$, the small 
parameter $\eps$, and the function $f$. 
In addition, we have here constants 
$c_0, c_1,c_2\geq0$, and we also recall that the definition of
the regularized Hilbert transform ${\mathcal H}$ is given in \eqref{Hilbert}.

\subsection{Abstract formulations}

\subsubsection{The case without charge nor current density} 
\label{sectFormulWithout}

We show in Appendix \ref{NDwithout} that the Maxwell equations 
can be put under the dimensionless form (\ref{Max3}), 
which itself has the form
\begin{equation}\label{eqgen}
\dsp \dt \bU +A(\partial)\bU+\frac{1}{\eps}E \bU+\eps^{1+p}  A_0 \bU=\eps  F(\eps,\bU),
\end{equation}
where 
$\bU$ is a $\R^n$ ($n\geq 1$) valued function depending on the time variable
$t$ and the space variable $x\in\R^d$ ($d\geq 1$),
$$
	\bU:\quad (t,x)\in \R\times\R^d\to \R^n. 
$$
The operator $A(\partial)$ is defined as
$$
	A(\partial)=\sum_{j=1}^d A_j\partial_j,
$$
where $\partial_j$ is the differentiation operator with respect to the $j$-th space coordinate. 
The matrix $A_0$ has size $n\times n$, $p$ is a positive number,  and the
following assumption is made on the matrices $A_j$ and $E$, and on the nonlinearity $F$. 
\begin{assumption}\label{assu1}
\item[(i)] The matrices $A_j$ ($j=1,\dots, d$) are constant coefficient $n\times n$, real valued, \emph{symmetric} matrices.
\item[(ii)] The matrix $E$ is a constant coefficient $n\times n$, real valued, \emph{skew symmetric} matrix.
\item[(iii)] There exists a smooth mapping $f: \R^+ \to \R$  vanishing at the origin, a real number $r> 0$, 
a quadratic form $Q:\C^n\mapsto \R^+$ and
a trilinear symmetric mapping $T:(\C^n)^3\to \C^n$ 
(with $T(\R^n\times\R^n\times\R^n)\subset\R^n$)
such that
$$
\forall U\in \C^n, \qquad F(\eps,U)= \big(1+ f\big(\eps^rQ(U))\big)\,T(U,\bar U, U).
$$
\end{assumption}
\begin{remark}
There exist of course situations where the leading order of 
the nonlinearity is not cubic (it can be quadratic for non centro-symmetric crystals for instance) or not of this form; since we are interested here in deriving variants of
the standard cubic nonlinear Schr\"odinger equation, we restrict ourselves to this framework for the sake
of simplicity. 
\end{remark}
\begin{example}\label{remM1}
As previously said, the dimensionless version (\ref{Max3}) 
of the Maxwell equations can be put under the form
(\ref{eqgen}) and they satisfy Assumption \ref{assu1} 
with $n=12$, $\bU=(\ttB,\ttE,\ttQ^\sharp,\ttP^\sharp)^T$. 
See Appendix \ref{sectexplcomp} for more details. 
 \end{example}

\subsubsection{The case with charge and current density}

As shown in Appendix \ref{NDwith}, the system \eqref{Max4} 
of Maxwell's equations with partial ionization can be put 
under the general form
\begin{equation}\label{eqgen2}
\left\{
\begin{array}{l}
\dsp \dt \bU + A(\partial)\bU + \frac{1}{\eps}E \bU
+ \eps^{1+p}  A_0 \bU = \\
\hspace{1cm}\dsp
\eps  F(\eps, \bU) - \eps {\mathcal H} 
\left({\eps\frac{{\bf k}}{k^2}\cdot D}\right)
({\bf W} C_1^T C_1 \bU) - \eps c \, C_1^T G(C_1\bU,\bW) , 
\vspace{1mm} \\
\dsp \dt \bW = \eps G(C_1\bU,\bW)\cdot C_1\bU , 
\end{array}
\right.
\end{equation}
where, as in \S~\ref{sectFormulWithout}, $\bU$ is a 
$\R^n$-valued function, whereas $\bW$ is an
$\R$-valued function of $(t,x)\in\R\times\R^d$. 
The matrices $A_j$ and $E$, as well as the nonlinearity $F$, 
satisfy Assumption~\ref{assu1}. 
Concerning the other coefficients of the system, we assume 
the following. 
\begin{assumption} \label{assu2}
\item[(i)] The real, constant matrix $C_1$  
has size $m \times n$ (with $m\in \N$).
\item[(ii)] The constant $c$ is positive.
\item[(iii)] There exists two real, positive constants $c_1$ and $c_2$
  and an integer $K\geq 1$ such that 
$$
\forall E\in \C^m, \, \forall w\in\C, \qquad 
G(E,w)=c_1\abs{E}^{2K-2}E + c_2 w E .
$$
\end{assumption}
\begin{remark}
As in Remark \ref{remM1} for the case without ionization, we can put
Maxwell's equation with ionization terms \eqref{Max4} under the
abstract form \eqref{eqgen2}. Using the same notations as in Remark
\ref{remM1}, the matrix $C_1$ is the projection matrix such that $C_1
{\bf U}=\ttE$, and ${\bf w}=\rho$.
\end{remark}

\subsection{The Cauchy problem}

We are considering initial conditions that correspond 
to laser pulses. In the case without charge nor current 
density (equation \eqref{eqgen}), they are 
fast oscillating \emph{wave packets} slowly modulated 
by an envelope, 
\begin{equation}\label{CI}
\bU_{\vert_{t=0}}=u^0(x)e^{i\frac{\bk\cdot x}{\eps}}+\cc,
\end{equation}
where $\bk\in\R^d$ is the (spatial) \emph{wave-number} 
of the oscillations. Taking charge and current density into account 
(equation \eqref{eqgen2}), we need to provide initial conditions for
$\bW$; since we are interested here in the situation where this
quantity is created by the laser when it reaches high intensities
near self-focusing, we take these initial conditions to be
initially zero\footnote{One could more generally and without 
supplementary difficulty consider non-oscillating initial conditions 
for the charge density $\bW$} for the sake of clarity.
\begin{equation}\label{CIVW}
\bW_{\vert_{t=0}}=0.
\end{equation}

The evolution equation (\ref{eqgen}) (as well as (\ref{eqgen2})) being of semilinear nature, it is natural to work with Banach algebra in view of a resolution by Picard iterations. Throughout this article, we assume
that 
$
u^0\in B$, with
$$
B=H^{t_0}(\R^d)^n,\quad (t_0>d/2) 
$$
or
$$
 B=W(\R^d)^n:=\{f\in {\mathcal S}'(\R^d)^n, \abs{f}_B:=\abs{\widehat{f}}_{L^1}<\infty\}
$$
 (the so called Wiener algebra, which is better adapted than $H^{t_0}(\R^d)^n$
to handle short pulses, see \cite{ColinLannes,LannesSP}). In both cases,  $B$ is stable by translations in Fourier space (this ensures that if $u^0\in B$ in (\ref{CI}) then $\bU_{\vert_{t=0}}\in B$) and is a Banach algebra in the sense that 
$$
\forall f,g\in B,\qquad f\cdot g\in B\quad\mbox{ and }\quad \abs{f\cdot g}_B\lesssim \abs{f}_B\abs{g}_B.
$$
For all $k\in \N$, we also define
$$
B^{(k)}=\{f\in B,\quad \forall \alpha\in \N^d,\quad \forall \abs{\alpha}\leq k,\quad \partial^\alpha f\in B\},
$$
endowed with its canonical norm.

We are interested in deriving asymptotics to the solution formed by (\ref{eqgen})-(\ref{CI}), with initial envelope $u^0\in B$, and more generally 
(\ref{eqgen2})-(\ref{CI})-\eqref{CIVW},  if we want to be able to handle ionization processes in nonlinear optics. This requires a further assumption on the nonlinearity $F$, namely that $F$ acts on $B$ and is locally Lipschitz.
\begin{assumption}\label{assuF}
In addition to (iii) of Assumption \ref{assu1}, the mapping $F$ satisfies, uniformly with respect to 
$\eps\in [0,1)$,
\item[(i)] For all $f\in B$, one has $F(f)\in B$ and
$$
\abs{F(\eps,f)}_B\leq C(\abs{f}_B)\abs{f}_B
$$
\item[(ii)] For all $f,g\in B$, one has
$$
\abs{d_f F(\eps,\cdot)g}_B\leq C(\abs{f}_B)\abs{g}_B.
$$
\end{assumption}
\begin{example}
When $B=H^{t_0}(\R^d)^n$ ($t_0>d/2$), Assumption \ref{assuF} 
is  always satisfied (by Moser's inequality); when $B=W(\R^d)^n$, 
the assumption holds for analytic nonlinearities.
\end{example}

\section{Derivation of NLS-type  equations}\label{sectDer}

The Schr\"odinger approximation takes into account the \emph{diffractive} effects that modify over large times the propagation
along rays of standard geometrical optics. These diffractive effects are of linear nature and are known \cite{DJMR,Kalyakin,JMR,Schneider,Lannes}
to appear for time
scales of order $O(1/\eps)$ for the initial value problem formed by the linear part of (\ref{eqgen}) and (\ref{CI}).
This is the reason why we are interested in proving the existence and describing the solutions to the (nonlinear) 
initial value problem (\ref{eqgen})-(\ref{CI})
over such time scales. 

For the sake of simplicity, the initial value problem (\ref{eqgen})-(\ref{CI}) (no ionization) is first considered. Up to minor modifications, the results of \S\S \ref{profileeq}, \ref{sectSVEA} and \ref{sectAs31} are known \cite{ColinLannes}; we reproduce them here because they are necessary steps to derive the family 
of NLS equations (\ref{NLSfamily1}) and, for the sake of clarity, their proof is sketched in a few words.
The general idea to derive the Schr\"odinger equations of \S \ref{sectAs33} with improved dispersion was introduced in \cite{ColinLannes} but the computations are carried further here. We then derive in \S \ref{sectselfsteep} a new class of models with derivative nonlinearity for which a local well-posedness result is proved. When applied to Maxwell's equations, these
derivative nonlinearities yield the so-called self-steepening operators;  to our knowledge, 
this is the first rigorous explanation of these terms.

The asymptotic description of  (\ref{eqgen2})-(\ref{CI}) 
(i.e. ionization is now included) 
is then addressed in \S \ref{sectNLSioniz}.
\subsection{The profile equation}\label{profileeq}

We show here that under reasonable assumptions on  $F$, solutions to the initial
value problem
(\ref{eqgen})-(\ref{CI}) exist for times of order $O(1/\eps)$ and that there
can be written under a very convenient form using  a \emph{profile} $U$,
\begin{equation}\label{eq2}
\bU(t,x)=U\left(t,x,\frac{\bk\cdot x-\omega t}{\eps}\right),
\end{equation}
with $U(t,x,\theta)$ periodic with respect to $\theta$ and for any $\omega\in \R$, provided
that $U$ solves the \emph{profile equation}
\begin{equation}\label{eq3}
\left\lbrace
\begin{array}{l}
	\dsp \dt U+A(\partial)U+\frac{i}{\eps}\cL(\omega D_\theta,\bk D_\theta)U+\eps^{1+p} A_0 U=\eps F(\eps, U),\vspace{1mm}\\
	\dsp U_\init(x,\theta)=u^0(x)e^{i\theta}+\cc.
	\end{array}\right.
\end{equation}
Here, we used the notation
\begin{equation}\label{eq4}
\cL(\omega D_\theta,\bk D_\theta)=-\omega D_\theta+A(\bk)D_\theta+\frac{E}{i},
\end{equation}
with $D_\theta=-i\partial_\theta$ and $A(\bk)=\sum_{j=1}^d A_j k_j$.

\begin{theorem}\label{th1}
Let $B=H^{t_0}(\R^d)^n$ or $B=W(\R^d)^n$ and $u^0\in B$. 
Under Assumptions \ref{assu1} and \ref{assuF}, 
there exists $T>0$ such that for all $0<\eps\leq1$ 
there exists a unique solution $\bU\in C([0,T/\eps];B)$ 
to (\ref{eqgen})-(\ref{CI}). 
Moreover, one can write $\bU$ under the form
$$
\bU(t,x)=U\left(t,x,\frac{\bk\cdot x-\omega t}{\eps}\right),
$$
where $U$ solves the \emph{profile equation} (\ref{eq3}).
\end{theorem}
\begin{proof}
The proof is a slight adaptation of the one given in \cite{ColinLannes} in the trilinear case; 
consequently, we just give
the main steps of the proof. Quite obviously, a solution $\bU$ 
to (\ref{eqgen})-(\ref{CI}) is given by (\ref{eq2}) if (\ref{eq3}) admits
a solution $U\in C([0,T/\eps];H^{k}(\T;B))$ ($k\geq 1$) where
\begin{equation} \label{defHk} 
H^{k}(\T;B)=\left\{ f=\sum_{n\in\Z} f_n e^{in\theta}, \abs{f}_{H^{k}(\T,B)}<\infty\right\}
\end{equation}
and with $\abs{f}^2_{H^{k}(\T,B)}=\sum_{n\in\Z}(1+n^2)^k\abs{f_n}^2_{B}$. For $k\geq 1$, $H^k(\T,B)$ is a Banach algebra; moreover the evolution operator $S(t)$ associated to the linear part of (\ref{eq3}),
$$
	S(t)=\exp\left(-tA(\partial)-\frac{i}{\eps}t\cL(\omega D_\theta,k D_\theta)\right)
$$
 is unitary on  $H^{k}(\T;B)$ (thanks to point (i) and (ii) of Assumption \ref{assu1}). One can 
therefore construct a (unique) solution to (\ref{eq3}) by a standard iterative scheme
$$
	U^{l+1}(t)=S(t)U^0 + \eps \int_0^t S(t-t')
	\big[ F(\eps,U^l)(t')-\eps^p A_0U^l \big]{\rm d}t',
$$
with $U^0=U_{init}$. 
Indeed, one has thanks to Assumption \ref{assuF},
$$
\abs{U^{l+1}(t)}_{H^k(\T;B)}\leq \abs{U^0}_{H^k(\T;B)}+\eps \int_0^t \big[\eps^p\abs{A_0 U^l}_{H^k(\T;B)}+C(\abs{U^l}_{H^k(\T;B)})\abs{U^l}_{H^k({\T;B})}\big];
$$
thanks to the $\eps$ in front of the integral. An estimate 
of the same kind is valid for a difference of iterates, 
by point (ii) of Assumption \ref{assuF}. By a fixed point 
argument, this ensures that  the sequence converges to 
a solution on $[0,T/\eps]$ for some $T>0$ 
\emph{independent of $\eps$}. 
Uniqueness then follows classically from an energy estimate on the difference of two solutions. 
\end{proof}

\subsection{The slowly varying envelope approximation}\label{sectSVEA}

The slowly varying envelope approximation (SVEA) consists in writing the profile $U$ under the form
\begin{equation}\label{apenv}
	U(t,x,\theta)\sim u_{env}(t,x)e^{i\theta}+\cc;
\end{equation}
plugging this approximation into the profile equation (\ref{eq3}) and keeping only
the first harmonic in the Fourier expansion yields easily (writing $u=\ue$),
$$
\dt u+A(\partial) u+\frac{i}{\eps}\cLk u+\eps^{1+p}  A_0 u=\eps  F^{env}(\eps, u),
$$
where
\begin{equation}\label{defFenv}
F^{env}(\eps, u)=\frac{1}{2\pi}
\int_0^{2\pi}e^{-i\theta}F(\eps, u e^{i\theta}+\cc) \, {\rm d}\theta.
\end{equation}
\begin{example}\label{excubic}
With $F(u)=\abs{u}^2 u$, one gets $F^{env}(u)=(u\cdot u)\overline{u}+2\abs{u}^2u$.
\end{example}
Denoting $D=-i\nabla$, we observe that
\begin{eqnarray*}
A(\partial)+\frac{i}{\eps}\cLk&=&A(\partial)+\frac{i}{\eps}(-\omega \Id+A(\bk))\\
&=&\frac{i}{\eps}(-\omega\Id+A(\bk+\eps D))\\
&:=&\frac{i}{\eps}\cL(\omega,\bk+\eps D),
\end{eqnarray*}
where the last notation is of course consistent with (\ref{eq4}).

As a consequence of these computations, we see that in order for (\ref{apenv}) to hold, it is necessary that 
$u=\ue$ satisfies the  \emph{envelope equation}
\begin{equation}
\label{eq18}
\left\lbrace
\begin{array}{l}
\dsp \dt u+\frac{i}{\eps}\cL(\omega,\bk+\eps D)u+\eps^{1+p}  A_0 u=\eps   F^{env}(\eps, u),\vspace{1mm}\\
\dsp u_\init=u^0.
\end{array}
\right.
\end{equation}

As implicitly assumed by omitting the fast oscillating scale in the argument of the envelope function $u_{env}(t,x)$, the envelope must not contain any fast oscillation. However, 
\begin{itemize}
\item The singular part of the linear term in (\ref{eq18}) creates fast oscillations with frequencies $\omega-\omega_j(\bk)$,
where the $\omega_j(\bk)$ stand for the eigenvalues of $\cL(0,\bk)$.
\item The nonlinearity creates other oscillations that may resonate with the linear propagator.
\end{itemize}

There is one way to avoid the catastrophic effects of these two scenarios. Choosing $\omega=\omega_j(\bk)$ for some $j$ and assuming that, up to $O(\eps)$ terms, the initial envelope $u^0$ is contained in the corresponding eigenspace
prevents the creation of oscillations by the linear propagator. This is the \emph{polarization condition}. The nonlinearity will however create harmonics of the main oscillation $\bk\cdot x-\omega_j(\bk)t$ and it is necessary to make a non resonance assumption. What is called \emph{characteristic variety} in the assumption below is the
set ${\mathcal C}_{\cL}\subset \R^{d+1}$ defined as
$$
{\mathcal C}_{\cL}=\{(\omega',\bk')\in \R^{1+d},\quad \det \cL(\omega',\bk')=0\}.
$$
Let us also recall that we assumed that the nonlinearity is under the form
$$
F(\eps,U)= \big(1+ f\big(\eps^rQ(U))\big)\,T(U,\bar U, U),
$$
with $f(0)=0$, $Q$ a quadratic form and $T$ a trilinear symmetric mapping. 
If $U$ is a monochromatic oscillation, the nonlinearity $\eps F(\eps,U)$ creates third harmonic with size $O(\eps)$, a fifth harmonic (if $f'(0)\neq 0$) with size $O(\eps^{1+r})$, etc. The non-resonance condition stated below holds for the $(2p+3)$-th harmonics, for all $p\geq 0$ such that $pr<1$ (the contribution of higher harmonics is small enough to be controlled even if it is resonant).
\begin{assumption}\label{assu3}
The characteristic variety $\CL$ and the frequency/wave number couple $(\omega,\bk)$
satisfy:
\begin{enumerate}
\item There exist $m$ functions $\omega_j \in C^\infty(\R^d\backslash\{0\})$ 
($j=1,\dots,m$) such that 
$$
	\CL\backslash\{0\}=\bigcup_{j=1}^m \big\{(\omega_j(\bk'),\bk'),\bk'\in\R^d\backslash\{0\}\big\};
$$
up to a renumbering, we assume that $(\omega,\bk)=(\omega_1(\bk),\bk)$.
\item There exists a constant $c_0>0$ such that
$$
	\inf_{\bk'\in\R^d}\abs{\omega-\omega_j(\bk')}\geq c_0,\qquad
	j=2,\dots m.
$$
\item (Non resonance assumption) One has $\pm (2p+3)(\omega,\bk)\notin {\mathcal C}_\cL$, for all $p\geq 0$ such that $pr<1$.
\end{enumerate}
\end{assumption}
\begin{notation}\label{nota2}
We denote by $\pi_j(\bk)$ ($j=1,\dots,m$) the eigenprojectors 
of the eigenvalues $\omega_j(\bk)$ of $A(\bk)+\frac{1}{i}E$; 
in particular, we have
$$	\cL(0,\bk)=A(\bk)+\frac{1}{i}E = 
\sum_{j=1}^m\omega_j(\bk)\pi_j(\bk).
$$
\end{notation}
\begin{example}
For Maxwell's equations, it is shown in Appendix 
\ref{sectexplcomp} that Assumption \ref{assu3} 
is satisfied with $m=7$ in dimension $d=3$, for $\omega\neq0$.
Explicit expressions for the eigenprojectors $\pi_j(\bk)$ 
are also provided in Appendix \ref{sectexplcomp}.
\end{example}
\begin{theorem}\label{th3}
Let $B=H^{t_0}(\R^d)^n$ or $B=W(\R^d)^n$ and $u^0\in B^{(1)}$, 
$r\in B$. 
Let Assumptions \ref{assu1}, \ref{assuF} and \ref{assu3} 
be satisfied and assume moreover that
$u^0=\pi_1(\bk)u^0+\eps r$. Then
\item[(i)] There exist $T>0$ and, for all $\eps\in(0,1]$, 
a unique solution $u\in C([0,T/\eps];B^{(1)})$ to (\ref{eq18}) 
with initial condition $u^0$.
\item[(ii)] There exists $\eps_0>0$ such that for all $0<\eps<\eps_0$, the solution $\bU$ to (\ref{eqgen})
provided by Theorem \ref{th1} exists on $[0,T/\eps]$ and
$$
\abs{\bU-\bU_{SVEA}}_{L^\infty([0,T/\eps]\times\R^d)}\leq 
\eps C(T,\abs{u^0}_B)(1+\abs{\nabla u^0}_B+\abs{r}_B),
$$
where $\bU_{SVEA}(t,x)=u(t,x) e^{i\frac{\bk\cdot x-\omega t}{\eps}}+\cc$.
\end{theorem}
\begin{proof}
Here again, the proof is a small variation of the one given in \cite{ColinLannes} for the trilinear case and $B=W(\R^d)^n$. We just indicate the main steps of the proof:\\
\emph{Step 1}. Existence and bounds of the solution $u$ to (\ref{eq18})  is established by a fixed point argument as in Theorem \ref{th1}.\\
\emph{Step 2}. We decompose $u$ as
$$
	u=u_1+u_{II}, \quad\mbox{ with }\quad u_{II}=\sum_{j=2}^m u_j,
 $$
and where $u_j=\pi_j(\bk+\eps D) u$ (see Notation \ref{nota2}).\\
\emph{Step 3.} Thanks to the assumption that $\omega=\omega_1(\bk)$ one gets from the equation obtained 
by applying $\pi_1(\bk+\eps D)$ to (\ref{eq18}) that 
$\abs{\dt u_1(t)}_B$ is uniformly bounded on $[0,T/\eps]$.\\
\emph{Step 4.} Using a non-stationary phase argument (on the semigroup formulation) relying on point (ii) of Assumption \ref{assu3} and the bound on $\dt u_1$ established in Step 3, 
and taking advantage of Assumption \ref{assuF},  we get that $\frac{1}{\eps}\vert u_{II}(t)\vert_B$ remains uniformly bounded on $[0,T/\eps]$.\\
\emph{Step 5.}  Using the non-resonance condition (iii) of Assumption \ref{assu3}, one can 
show that the third and higher harmonics created by the nonlinearity remain of order $O(\eps)$.
More precisely, the solution  $U\in H^1(\T;B)$ to (\ref{eq3}) provided by Theorem \ref{th1} can
be written as 
$$
	U(t,x,\theta)=U_{\rm app}(t,x,\theta)+\eps V(t,x,\theta),
$$
where $U_{\rm app}(t,x,\theta)=u(t,x)e^{i\theta}+\cc$, and 
$V$ remains bounded (with respect to $\eps$) in $C([0,T/\eps];H^1(\T;B)^n)$.  \\
\emph{Step 6.} Since $U(t)-U_{\rm app}(t)=\eps V(t)$, 
it follows from the above that 
$$
	\sup_{t\in [0,T/\eps]}
	\vert U(t)-U_{\rm app}(t)\vert_{H^1(\T;B)}
	\leq
	\eps C(T,\vert u^0\vert_B)(1+\vert\nabla u^0\vert_B+\abs{r}_B),
$$
and the theorem follows therefore from the observation that
$$
	\vert \bU-\bU_{SVEA}\vert_{L^\infty([0,T/\eps]\times\R^d)}
\leq \sup_{t\in [0,T/\eps]}
	\vert U(t)-U_{\rm app}(t)\vert_{H^1(\T;B)}.
$$
\end{proof}

\subsection{The Full Dispersion model}\label{sectAs31}

The idea is to diagonalize (\ref{eq18}) in order to work with a scalar equation. We project therefore
(\ref{eq18}) onto the eigenspace corresponding to the oscillating term. These naturally leads to introduce
$$
u_{FD}=\pi_1(\bk+\eps D)\ue.
$$
which naturally leads to the the \emph{full dispersion} scalar\footnote{The linear propagator is a scalar operator but the equation remains vectorial because of the nonlinear term. Indeed, $\pi_1(\bk)$ is in general not of rank $1$ (i.e. $\omega_1(\bk)$ is in general not of multiplicity one)} equation (writing $u=u_{FD}$)
\begin{equation}
	\label{eqfull}
	\left\lbrace
	\begin{array}{l}
	\!\dt u+\frac{i}{\eps}(\omega_1(\bk\!+\!\eps D)\!-\!\omega)u
	+\eps^{1+p} \pi_1(\bk\!+\!\eps D) A_0 u=\eps \pi_1(\bk\!+\!\eps D)F^{env}(\eps,u)\\
	\!u\,\init(x)=u^0(x)
	\end{array}\right.
\end{equation}
and with $\omega_1(\cdot)$ as in Assumption \ref{assu3}.\\
The following corollary shows that the full dispersion scalar
equation yields an approximation of the same precision as
the envelope equation for times $t\in [0,T/\eps]$.
\begin{coro}[Full dispersion model]\label{corofull}
	Under the assumptions of Theorem \ref{th3},  
\item[(i)] There exists $T>0$ and, for all $\eps\in(0,1]$,  
a unique solution $u\in C([0,T/\eps];B^{(1)})$ to (\ref{eqfull}) 
with initial condition $u^0$.
\item[(ii)]  There exists $\eps_0>0$ such that for all $0<\eps<\eps_0$, the solution $\bU$ to (\ref{eqgen})
provided by Theorem \ref{th1} exists on $[0,T/\eps]$ and
$$
\abs{\bU-\bU_{FD}}_{L^\infty([0,T/\eps]\times\R^d)}\leq 
\eps C(T,\abs{u^0}_B)(1+\abs{\nabla u^0}_B+\abs{r}_B),
$$
where $\bU_{FD}(t,x)=u(t,x) e^{i\frac{\bk\cdot x-\omega t}{\eps}}+\cc$.
\end{coro}
\begin{remark}
Equation (\ref{eqfull}) does not correspond exactly to the ``full dispersion'' model of \cite{ColinLannes,LannesSP},
 where the r.h.s is $\eps\pi_1(\bk)F^{env}(\eps,u)$ rather than  $\eps\pi_1(\bk+\eps D)F^{env}(\eps,u)$
(but it can be found as an ``intermediate model'' in \cite{CGL}). This change does
not affect the estimate given in the Corollary, but it is important to keep track of the frequency dependence of
the polarization of the nonlinear term to introduce the ``self-steepening'' operators in \S \ref{sectselfsteep}. Note also that the "full dispersion" model is related to the so-called "unidirectional pulse propagation equation" used in nonlinear optics \cite{KolesikMoloney,BergeSkupin2}.
\end{remark}
\begin{proof}
This is actually a byproduct of the proof of Theorem \ref{th3} since $u_{FD}$ coincides with $u_1$ in Step 2 of the
proof of Theorem \ref{th3}.
\end{proof}

\subsection{The nonlinear Schr\"odinger (NLS) equation}\label{sectAs32}

The NLS equation is easily derived from (\ref{eqfull}) using Taylor expansions
\begin{eqnarray*}
\frac{i}{\eps}(\omega_1(\bk+\eps D)-\omega)&=&
\cg\cdot \nabla-\frac{i}{2}\nabla\cdot H_\bk\nabla +O(\eps^2),\\
\pi_1(\bk+\eps D)&=&\pi_1(\bk)+O(\eps),
\end{eqnarray*}
where  $\cg=\nabla\omega_1(\bk)$ and $H_\bk$ stands for the Hessian of $\omega_1$ at $\bk$. Neglecting the $O(\eps^2)$
terms in (\ref{eqfull}) we define the $NLS$ approximation $u=u_{NLS}$ as the solution to
\begin{equation}\label{NLSp}
\left\lbrace\begin{array}{l}
 \partial_t u+\cg\cdot\nabla u
-\eps \frac{i}{2}\nabla\cdot H_\bk\nabla u+\eps^{1+p} \pi_1(\bk) A_0 u=
\eps  \pi_1(\bk) F^{env}(\eps, u).\\
 u\,\init(x)=u^0(x)
	\end{array}\right.
\end{equation}
We then get easily (see \cite{ColinLannes,LannesSP}) the following justification of the NLS approximation.
\begin{coro}[Schr\"odinger approximation]\label{coroschrod} 
Under the assumptions of Theorem \ref{th3}, one has 
for all $u^0\in B^{(3)}$ such that $u_0=\pi_1(\bk)u_0$,
\item[(i)] There exists $T>0$ and, for all $\eps\in(0,1]$,  
a unique solution $u\in C([0,T/\eps];B^{(3)})$ to (\ref{NLSp}) 
with initial condition $u^0$.
\item[(2)] There exists $\eps_0>0$ and ${\mathfrak c}_{NLS}>0$ 
such that for all $0<\eps<\eps_0$, the solution $\bU$ to 
(\ref{eqgen}) provided by Theorem \ref{th1} exists 
on $[0,T/\eps]$ and
$$
\abs{\bU-\bU_{NLS}}_{L^\infty([0,T/\eps]\times\R^d)}
\leq \eps C(T,\abs{u^0}_B)(1+\abs{\nabla u^0}_B
+{\mathfrak c}_{NLS}\abs{u^0}_{B^{(3)}}),
$$
where $\bU_{NLS}(t,x)=u(t,x) e^{i\frac{\bk\cdot x-\omega t}{\eps}}+\cc$.
\end{coro}
\begin{remark}\label{rembad}
The component ${\mathfrak c}_{NLS}\abs{u^0}_{B^{(3)}}$ in the error estimate of the Corollary is due to the
bad frequency behavior of the Schr\"odinger equation when the envelope of the oscillations ceases to 
be well localized in frequency. This is for instance the case for short pulses, chirped pulses (\cite{CGL,ColinLannes,LannesSP}), and near the focusing point. To describe such extreme situations, the standard NLS approximation does
a poor job, and this is why various variants have been derived in physics (e.g. \cite{BergeSkupin}).
\end{remark}
\begin{remark}
We assumed here that the polarization of the initial condition is exact (i.e. $r=0$ in Theorem \ref{th3}) for the sake
of simplicity; indeed, the solution remains polarized along $\pi_1(\bk)$ for all times and computations on real
physical models are much easier.
\end{remark}
\begin{example}\label{exrad}
In the frequent case where $\omega_1(\cdot)$ has a radial symmetry, and writing with a slight abuse of notation
$\omega_1(\bk)=\omega_1(k)$, with $k=\abs{\bk}$, we can write,
$$
\cg=\omega_1'(k){\bf e}_z,\qquad H_\bk=\frac{\omega_1'(k)}{\abs{\bk}}(I-{\bf e}_z\otimes{\bf e}_z)+{\bf e}_z\otimes{\bf e}_z \omega_1''(k),
$$
where we assumed without loss of generality that $(0z)$ is the direction of the wave number $\bk$, $\bk=k{\bf e}_z$.
In particular, (\ref{NLSp}) reads
$$
\dt u+\omega_1'(k)\dz u-\eps\frac{i}{2}\frac{\omega'_1(k)}{k}\Delta_\perp u-i\frac{\eps}{2}\omega_1''(k)\dz^2 u+\eps^{1+p} \pi_1(\bk) A_0 u=
\eps \pi_1(\bk) F^{env}(\eps, u),
$$
where $\Delta_\perp=\partial_x^2+\partial_y^2$ is the Laplace operator in transverse variables.
If we write $v(t,x,z)=u(\eps t,x,z-\omega_1'(k)t)$, we get 
$$
\partial_t v-\frac{i}{2}\frac{\omega'_1(k)}{k}\Delta_\perp v-i\frac{1}{2}\omega_1''(k)\dz^2 v+\eps^p \pi_1(\bk)A_0 v=
 \pi_1(\bk) F^{env}(\eps, v).
$$
\end{example}

\subsection{The nonlinear Schr\"odinger equation with improved dispersion
relation}\label{sectAs33}

We propose here to investigate further the Schr\"odinger equation with improved dispersion relation derived in \cite{ColinLannes}. As said in Remark \ref{rembad}, the NLS approximation has too bad dispersive properties to capture correctly the envelope of oscillating solutions to Maxwell's equations in extreme situations, where high frequencies are
released. Indeed, the dispersion relation $\omega_1(\cdot)$ of (\ref{eqgen}) is approximated by the second order polynomial
\begin{equation}\label{reldispNLS}
\omega_{NLS}(\bk')=\omega_1(\bk)+\cg\cdot (\bk'-\bk)+\frac{1}{2}(\bk'-\bk)\cdot H_\bk(\bk'-\bk).
\end{equation}
For Maxwell's equations and in dimension $d=1$,  
Figure \ref{aproxreldisp} shows that this dispersion relation 
is very poor when $\bk'$ is not close to $\bk$. 
\begin{figure}[ht!]
\begin{minipage}[t]{0.48\textwidth}
\begin{center}\includegraphics[width=\textwidth]{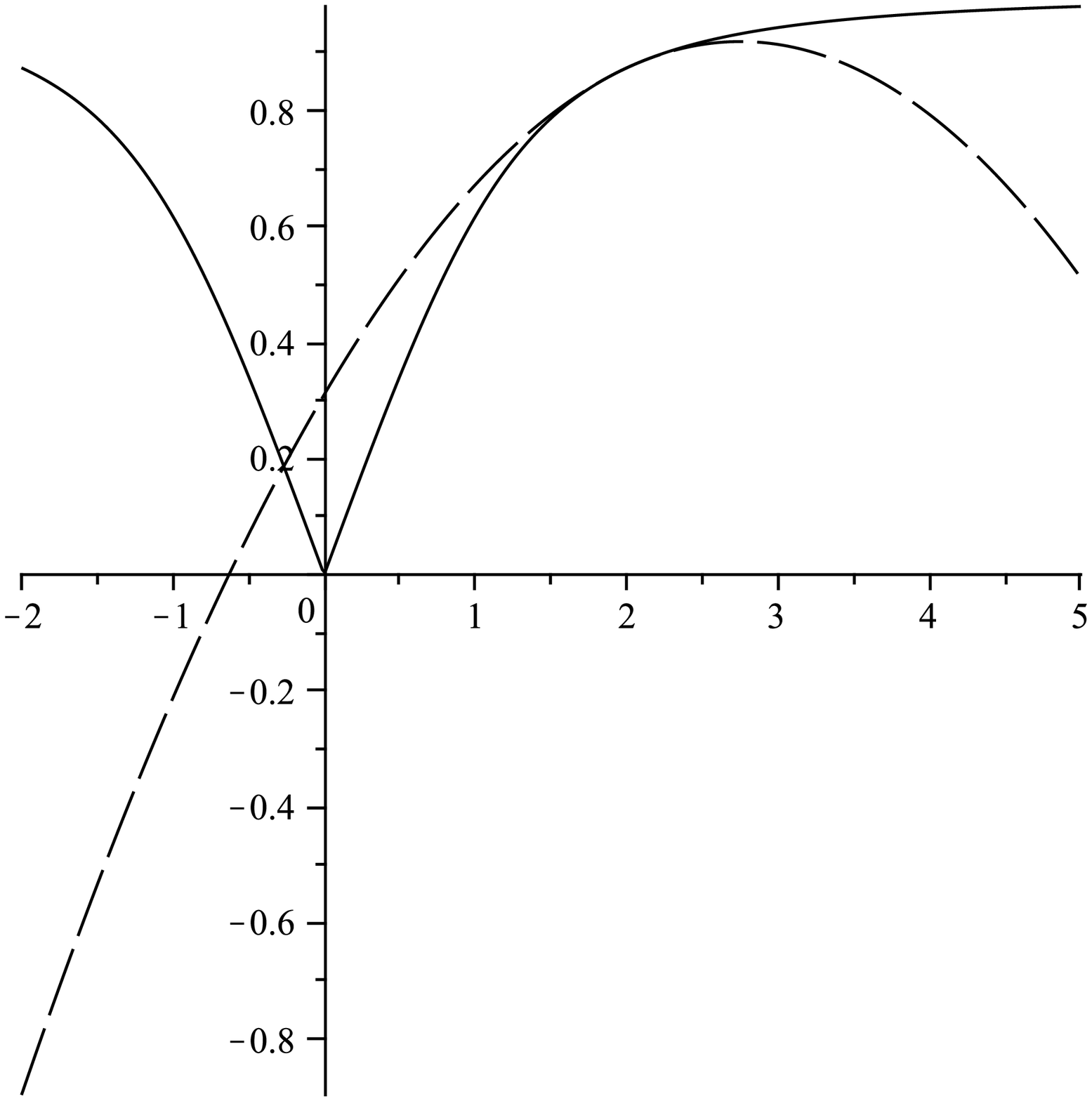}\end{center}
\end{minipage}
\begin{minipage}[t]{0.48\textwidth}
\begin{center}\includegraphics[width=\textwidth]{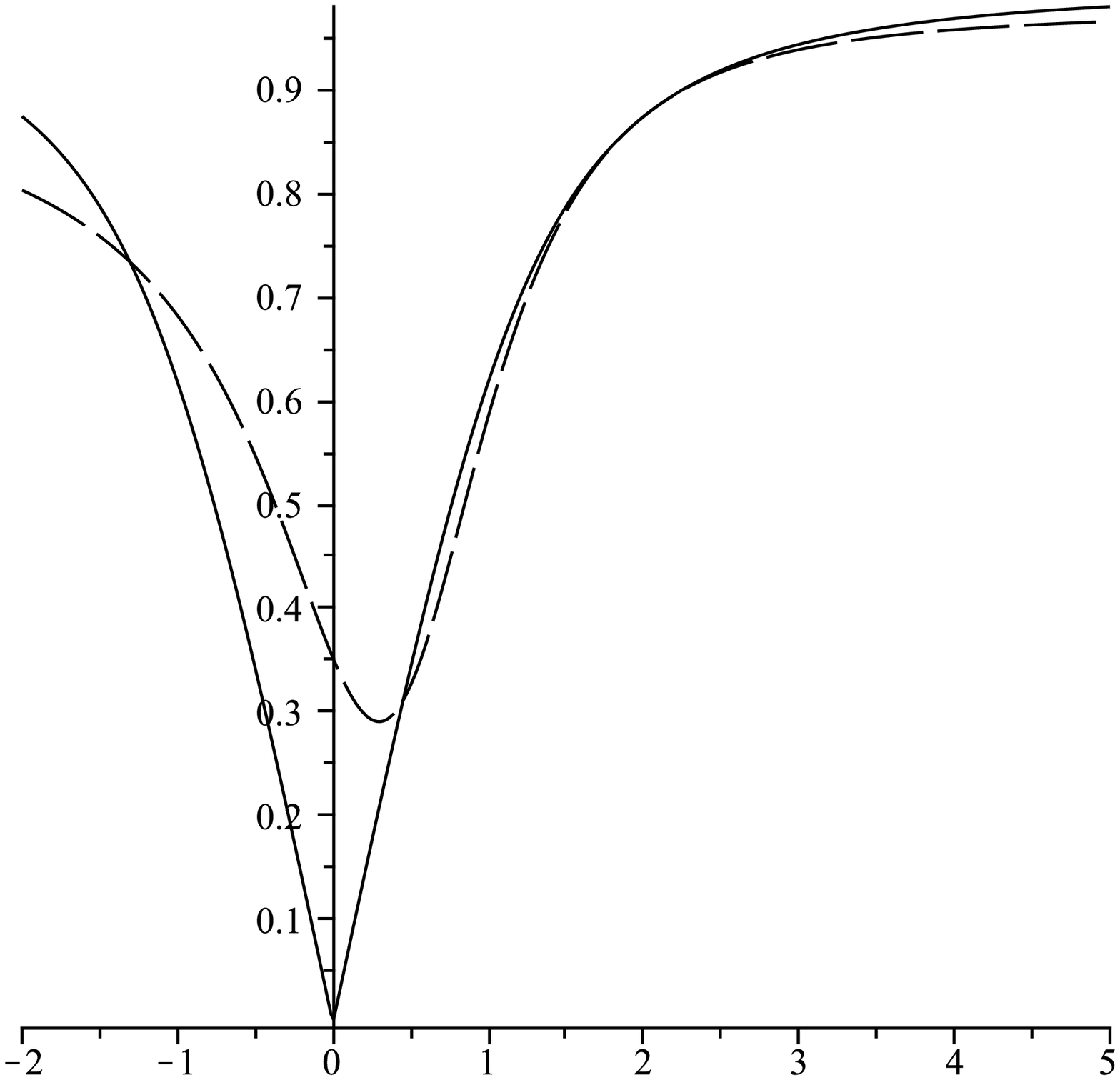}\end{center}
\end{minipage}
\caption{The $\omega_{+,-}$  component 
(see Appendix \ref{sectexplcomp}) of the characteristic 
variety of Maxwell's equation (solid) and the dispersion 
relation corresponding to the Schr\"odinger approximation 
(left) and improved Schr\"odinger (right). 
Here ${\bf k}=2$.}
\label{aproxreldisp}
\end{figure}
The idea introduced in \cite{ColinLannes} is to replace the \emph{linear} part of the Schr\"odinger approximation
by a linear operator that differs from the linear part of the Schr\"odinger approximation by $O(\eps^2)$ terms
only, but whose dispersion relation is far better.

More precisely, we consider an approximation under the form
\begin{equation}
	\label{uapp3}
	\bU_{imp}(t,x)= 
	u_{imp}(t,x)e^{i\frac{\bk\cdot x-\omega t}{\eps}}+\cc,
\end{equation}
where $u= u_{imp}$ solves the \emph{nonlinear Schr\"odinger equation
with improved dispersion relation}
\begin{equation}
	\label{eqpade}
	\left\lbrace
	\begin{array}{l}
	\dsp \big(1-i\eps{\bf b}\cdot\nabla-\eps^2\nabla\cdot B\nabla\big)
	\dt u\\
	\dsp\indent
	+\cg\cdot\nabla u-\eps\frac{i}{2} \nabla\cdot \big( H_\bk  +2 
        \cg\otimes{\bf b} \big)\nabla u+\eps^2 C_3(\nabla) u
	\\
	\dsp \indent+ \eps^{1+p} \pi_1(\bk) A_0 u=\eps  \pi_1(\bk) F^{env}(\eps, u)\\
	u\,\init(x)=u^0(x),
	\end{array}\right.
\end{equation}
where ${\bf b}\in \C^d$, $B\in {\mathcal M}_{d\times d}(\R)$ and $C_3(\nabla)$ is a third order homogeneous differential operator. We assume moreover that
\begin{equation}\label{hyp}
B\mbox{ is symmetric positive}, \quad {\bf b}\in Range(B),
	\quad\mbox{ and }\quad
	4-{\bf b}\cdot (B^{-1}{\bf b})>0
\end{equation}
(note that even though $B^{-1}{\bf b}$ is not unique when $B$ is not definite,
the scalar ${\bf b}\cdot (B^{-1}{\bf b})$ is uniquely defined). These assumptions ensure that the operator 
$(1-i\eps{\bf b}\cdot\nabla-\eps^2\nabla\cdot B\nabla)$ is invertible. This new model allows one to 
replace (\ref{reldispNLS}) by
$$
\omega_{imp}(\bk')=\omega_1(\bk)+\frac{\cg\cdot(\bk'-\bk)+\frac{1}{2}(\bk-\bk')\cdot (H_\bk+2\cg\otimes {\bf b})(\bk'-\bk)-C_3(\bk'-\bk)}{1+{\bf b}\cdot (\bk'-\bk)+(\bk'-\bk)\cdot B(\bk'-\bk)}.
$$
A good choice of $\bf b$, $B$ and $C_3$ allows a much better approximation of $\omega_1(\cdot)$, as shown in Figure \ref{aproxreldisp} for Maxwell's equation in dimension $d=1$.

Exactly as for Corollary \ref{coroschrod} we get the following result, where the only difference in the
error estimate with respect to Corollary \ref{coroschrod} is the constant ${\mathfrak c}_{imp}$ (with is much
smaller than ${\mathfrak c}_{NLS}$ for good choices of ${\bf b}$, $B$ and $C$). We refer to \cite{ColinLannes}
for the proof and numerical validations of this model for the approximation of short pulses and chirped pulses.
\begin{coro}[Schr\"odinger approximation with improved dispersion]\label{coroimp}
Under the assumptions of Theorem \ref{th3}, one has, 
for all $u^0\in B^{(3)}$ such that $\pi_1(\bk)u_0=u_0$,
\item[(i)] There exists $T>0$ and, for all $\eps\in(0,1]$,  
a unique solution $u\in C([0,T/\eps];B^{(3)})$ to (\ref{eqpade}) 
with initial condition $u^0$.
\item[(2)]  There exists $\eps_0>0$ and ${\mathfrak c}_{imp}>0$ such that for all $0<\eps<\eps_0$, the solution $\bU$ to (\ref{eqgen})
provided by Theorem \ref{th1} exists on $[0,T/\eps]$ and
$$
\abs{\bU-\bU_{imp}}_{L^\infty([0,T/\eps]\times\R^d)}
\leq \eps C(T,\abs{u^0}_B)(1+\abs{\nabla u^0}_B
+{\mathfrak c}_{imp}\abs{u^0}_{B^{(3)}}),
$$
where $\bU_{imp}(t,x)=u(t,x) e^{i\frac{\bk\cdot x-\omega t}{\eps}}+\cc$.
\end{coro}

\begin{example}\label{exrad2}
In the framework of Example \ref{exrad}, i.e. if  $\omega_1(\bk)=\omega_1(k)$ with $k=\abs{\bk}$ and $\bk=k{\bf e}_z$,
(\ref{eqpade}) can be written
\begin{eqnarray*}
 (1-i\eps {\bf b}\cdot \nabla -\eps^2\nabla\cdot  B\nabla)\dt u+\omega_1'(k)\dz u-\eps\frac{i}{2}\big(\frac{\omega'_1(k)}{k}\Delta_\perp +\omega_1''(k)\dz^2\big) u\\
=-\eps i\omega_1'(k){\bf b}\cdot \nabla\dz u+\eps^2C_3(\nabla)u+\eps^{1+p} \pi_1(\bk) A_0 =
\eps \pi_1(\bk) F^{env}(\eps, u).
\end{eqnarray*}
If we write $v(t,x,z)=u(\eps t,x,z-\omega_1'(k)t)$ and choose $C_3(\nabla)=-\omega_1'(k) \nabla\cdot B\nabla\dz$,
we get 
\begin{eqnarray}
\nonumber
(1-i\eps {\bf b}\cdot \nabla -\eps^2\nabla\cdot  B\nabla)\partial_t v-\frac{i}{2}\big(\frac{\omega'_1(k)}{k}\Delta_\perp +\omega_1''(k)\dz^2\big) v+\eps^p \pi_1(\bk)A_0 v\\
\label{eqrad2}
 =\pi_1(\bk) F^{env}(\eps, v).
\end{eqnarray}
A similar equation has been proposed in \cite{Lannes_book} \S 8.5.3 in the
framework of water waves equations.
\end{example}

\subsection{The NLS equation with frequency dependent polarization}\label{sectselfsteep}

The idea here is to improve the rough approximation $\pi_1(\bk+\eps D)\sim \pi_1(\bk)+O(\eps)$
used to derive the NLS approximation (see \S \ref{sectAs32}). Indeed, when the description of the envelope
of the laser pulse requires a broad band of frequencies as in the situations mentioned in \S \ref{sectAs33},
the variations of the polarization term $\pi_1(\bk+\eps D)$ in front of the nonlinearity in (\ref{eqfull})
become important and should be taken into account. We therefore make the following approximation,
$$
\pi_1(\bk +\eps D)\sim(1-i\eps {\bf b}\cdot\nabla-\eps^2\nabla\cdot B\nabla)^{-1}\big[\pi_1(\bk)+\eps  \pi_1'(\bk)\cdot D-i\eps ({\bf b}\cdot \nabla)\pi_1(\bk)\big],
$$
where ${\bf b}$ and $B$ are the same as in the NLS approximation with improved dispersion (\ref{eqpade}). In particular, if ${\bf b}=0$ and $B=0$ (standard NLS), then the above approximation coincides with the
first order Taylor expansion
$$
\pi_1(\bk +\eps D)=\pi_1(\bk)+\eps  \pi_1'(\bk)\cdot D.
$$
The general formula has the same accuracy as this Taylor expansion 
as $\eps\to 0$. The corresponding approximation is given by
\begin{equation}
	\label{uapp4}
	\bU_{pol}(t,x)= 
	u_{pol}(t,x)e^{i\frac{\bk\cdot x-\omega t}{\eps}}+\cc,
\end{equation}
where $u= u_{pol}$ solves the \emph{nonlinear Schr\"odinger equation
with frequency dependent polarization}
\begin{equation}
	\label{eqpolar}
	\left\lbrace
	\begin{array}{l}
	\dsp \big(1-i\eps{\bf b}\cdot\nabla-\eps^2\nabla\cdot B\nabla\big)
	\dt u\\
	\dsp\indent
	+\cg\cdot\nabla u-\eps\frac{i}{2} \nabla\cdot \big( H_\bk  +2 
        \cg\otimes{\bf b} \big)\nabla u+\eps^2 C_3(\nabla) u
	+ \eps^{1+p} \pi_1(\bk) A_0 u\\
        \dsp \indent\indent =\eps  \big[\pi_1(\bk)+\eps  \pi_1(\bk)\pi_1'(\bk)\cdot D-i\eps ({\bf b}\cdot \nabla)\pi_1(\bk)\big] F^{env}(\eps, \pi_1(\bk)u)\\
	u\,\init(x)=u^0(x),
	\end{array}\right.
\end{equation}
where ${\bf b}$, $B$ and $C_3(\nabla)$ are the same as in (\ref{eqfull}). 

Contrary to all the previous models, the nonlinearity in (\ref{eqpolar}) seems to be of quasilinear rather than
nonlinear nature. It turns out however that the presence of the operator $\big(1-i\eps{\bf b}\cdot\nabla-\eps^2\nabla\cdot B\nabla\big)$ in front of the time derivative plays a smoothing role allowing the control of one or several first order derivatives (see the discussion in point (\ref{offaxis}) in p. \pageref{offaxis}). If the first order
derivatives involved in the nonlinearity are all controlled by this smoothing operator, then the nonlinearity remains
semilinear in nature. As shown in the proof below, the component $-i\eps ({\bf b}\cdot \nabla)\pi_1(\bk) F^{env}(\eps, \pi_1(\bk)u)$ of the nonlinearity is always semilinear in this sense. This is not the case for the component
$\eps  \pi_1(\bk)\pi_1'(\bk)\cdot D F^{env}(\eps, \pi_1(\bk)u)$ that may be of quasilinear nature, in which
case a symmetry assumption is needed on the nonlinearity to ensure
local well-posedness. In order to state this assumption, it is
  convenient to introduce the norm $\abs{\cdot}_*$ defined as
$$
\abs{u}_*^2=\big(u, (1-i\eps {\bf b}\cdot \nabla-\eps^2\nabla\cdot B\nabla)u\big).
$$
\begin{assumption}\label{assQL}
For all $v\in W^{1,\infty}(\R^d)^n$ and $u\in L^2(\R^d)^n$ such that $\pi_1(\bk)u=u$, one has
$$
\forall 1\leq j\leq d,
\qquad \Re\Big(\pi_1(\bk)\pi_1'(\bk)\cdot {\bf e}_j 
d_{{v}} F^{env} D_j u,u\Big)\leq 
\cst \abs{v}_{W^{1,\infty}}\abs{u}_*^2,
$$
where ${\bf e}_j$ is the unit vector in the $j$-th direction of $\R^d$ and $d_vF^{env}$ is the derivative at $v$ of the mapping $u\mapsto F^{env}(\eps,u)$.
\end{assumption} 
\begin{example}
The computations performed in Appendix \ref{sectexplcomp} show that this assumption is satisfied by the dimensionless Maxwell equations (\ref{MaxND}).
\end{example}

The approximation furnished by (\ref{eqpolar}) is 
justified by the following corollary (the difference in the estimate with respect to Corollary \ref{coroimp}
is just a better nonlinear constant, denoted $C_{pol}$ to insist on this point). For the sake of simplicity, 
we take $B=H^{t_0}(\R^d)^n$ ($t_0>d/2$) here, but adaptation to Wiener spaces are possible.
\begin{coro}[Schr\"odinger approximation with frequency dependent polarization]\label{coropol} 
Let the assumptions of Theorem \ref{th3} be satisfied 
and assume moreover that Assumption \ref{assQL} 
holds with $B=H^{t_0}(\R^d)$ ($t_0>d/2$). 
Then for all $u_0\in B^{(3)}$ such that $\pi_1(\bk)u_0=u_0$, one has
\item[(i)] There exists $T>0$ and, for all $\eps\in(0,1]$,  
a unique solution $u\in C([0,T/\eps];B^{(3)})$ to (\ref{eqpolar}) 
with initial condition $u^0$.
\item[(2)]  There exists $\eps_0>0$ and ${\mathfrak c}_{imp}>0$ such that for all $0<\eps<\eps_0$, the solution $\bU$ to (\ref{eqgen})
provided by Theorem \ref{th1} exists on $[0,T/\eps]$ and
$$
\abs{\bU-\bU_{pol}}_{L^\infty([0,T/\eps]\times\R^d)}
\leq \eps C_{pol}(T,\abs{u^0}_B)(1+\abs{\nabla u^0}_B
+{\mathfrak c}_{imp}\abs{u^0}_{B^{(3)}}),
$$
where $\bU_{pol}(t,x)=u(t,x) e^{i\frac{\bk\cdot x-\omega t}{\eps}}+\cc$.
\end{coro}
\begin{remark}
We have introduced the variation of the polarization on the NLS equation with improved dispersion (\ref{eqpade}), but the two steps are independent (i.e., one can take ${\bf b}=0$, $B=0$ and $C_3(\nabla)=0$
in (\ref{eqpolar})).
\end{remark}
\begin{remark}
Note that in (\ref{eqpolar}), we have applied $\pi_1(\bk)$ to the full nonlinearity (hence the term $\pi_1(\bk)\pi_1'(\bk)\cdot D$ instead of $\pi_1'(\bk)\cdot D$). This means that we keep the effect of the frequency dependent polarization on the main polarized space Range$(\pi_1(\bk))$ only. Similarly, we have replaced $F^{env}(\eps,u)$ by $F^{env}(\eps,\pi_1(\bk)u)$. This latter substitution would not change anything to the previous NLS models since we have seen that polarized initial conditions remain polarized. Its purpose in (\ref{eqpolar}) is to make Assumption \ref{assQL} much easier to check.
\end{remark}
\begin{proof}
The justification of (\ref{eqpolar}) is performed as for the other models. The only difference here is that
local well-posedness for a time scale of order $O(1/\eps)$ must be established. We therefore show here that
such a local well-posedness result holds if $u_0\in H^{s+1}(\R^d)^n$ with $s>t_0+1$.
We just prove a priori energy estimates for (\ref{eqpolar}); existence, uniqueness and stability with respect
to perturbations can be deduced classically.\\
The natural energy associated to (\ref{eqpolar}) is given for all $s\geq 0$ by
$$
E^s(u)=\frac{1}{2}\big( (1-i\eps{\bf b}\cdot\nabla-\eps^2\nabla\cdot B\nabla)\Lambda^s u,\Lambda^s u\big)=\frac{1}{2}\abs{\Lambda^su}_*^2.
$$
Under the assumptions (\ref{hyp}) on $B$ and ${\bf b}$, $E^s(u)^{1/2}$ defines a norm that controls uniformly the $H^s$-norm. It may also control a certain number of first order derivatives. The important point for the local well-posedness
of (\ref{eqpolar}) is that it always controls first order derivatives in the direction ${\bf b}\cdot \nabla$; more precisely,
we claim that there exists $c>0$ independent of $\eps$ such that for all $u$, 
\begin{equation}\label{eneq}
E^s(u)\geq c\big(\abs{u}_{H^s}^2+\eps^2 \abs{{\bf b}\cdot \nabla u}_{H^s}^2\big).
\end{equation}
A quick look on the Fourier side shows that it is equivalent to prove that
$$
\forall X\in \R^d, \qquad 1+{\bf b}\cdot X+X\cdot BX\geq c(1+{\bf b}\cdot X)^2,
$$
which is a consequence of (\ref{hyp}).\\
Multiplying $\Lambda^s$(\ref{eqpolar}) by $\Lambda^s \bar u$ and integrating by parts, we get
\begin{eqnarray*}
\frac{d}{\rm dt}E^s(u)&=&-\eps^{1+p}\Re\big(\pi_1(\bk)A_0\Lambda^s u,\Lambda^s u\big)+
\eps \big(\pi_1(\bk)\Lambda^s F(u),\Lambda^s u\big)\\
& &+\eps^2 \sum_{j=1}^d\Re\big(\pi_1(\bk)\pi_1(\bk)'\cdot {\bf e}_j \Lambda^s d_u F D_j u,\Lambda^s u\big)\\
& &-\eps^2\Re\big(i\pi_1(\bk)\Lambda^s ({\bf b}\cdot \nabla)F(u),\Lambda^s u\big)\\
&:=& I_1+I_2+I_3+I_4.
\end{eqnarray*}
where we denoted $F(u)=F^{env}(\eps,\pi_1(\bk)u)$. It is straightforward to control $I_1$, and Moser's estimate
gives a control of $I_2$,
$$
I_1\lesssim \eps^{1+p}\abs{u}_{H^s}^2,\qquad I_2\leq \eps C(\abs{u}_\infty)\abs{u}_{H^s}^2.
$$
In order to control $I_3$, we must split it into two parts,
\begin{eqnarray*}
I_3&=&\eps^2 \sum_{j=1}^d\Re\big(\pi_1(\bk)\pi_1(\bk)'\cdot {\bf e}_j  d_u F D_j \Lambda^s u,\Lambda^s u\big)\\
& &+\eps^2 \sum_{j=1}^d\Re\big(\pi_1(\bk)\pi_1(\bk)'\cdot {\bf e}_j [\Lambda^s, d_u F] D_j u,\Lambda^s u\big);
\end{eqnarray*}
the first component is controlled using Assumption \ref{assQL} while the Kato-Ponce and Moser estimate give a
control of the second one,
$$
I_3\leq \eps^2 \cst \abs{u}_{W^{1,\infty}}\abs{u}_{H^s}^2+\eps^2C(\abs{u}_{W^{1,\infty}})\abs{u}_{H^s}^2.
$$
Remarking that $\Lambda^s({\bf b}\cdot \nabla) F(u)=\Lambda^s d_uF ({\bf b}\cdot \nabla)u$ and using the 
tame product estimate $\abs{fg}_{H^s}\lesssim \abs{f}_{L^\infty}\abs{g}_{H^s}+\abs{f}_{H^s}\abs{g}_{L^\infty}$
and Moser's inequality, we get 
$$
I_4\leq \eps C(\abs{u}_{W^{1,\infty}})(\abs{u}_{H^s}+\eps \abs{{\bf b}\cdot \nabla u}_{H^s})\abs{u}_{H^s}.
$$
Gathering all the above estimates and using (\ref{eneq}), we obtain
$$
\frac{d}{\rm dt}E^s(u)\leq \eps C(\abs{u}_{W^{1,\infty}})E^s(u).
$$
Since moreover $\abs{u}^2_{W^{1,\infty}}\lesssim E^s(u)$ for $s>d/2+1$, we deduce from Gronwall's estimate
that for such $s$, the energy $E^s(u)$ remains bounded for times of order $O(1/\eps)$.
\end{proof}
\begin{example}
In the framework of Examples \ref{exrad} and \ref{exrad2}, we can check that
 $v(t,x,z)=u(\eps t,x,z-\omega_1'(k)t)$ solves
\begin{eqnarray}
\nonumber
(1-i\eps {\bf b}\cdot \nabla -\eps^2\nabla\cdot  B\nabla)\partial_t v-\frac{i}{2}\big(\frac{\omega'_1(k)}{k}\Delta_\perp +\omega_1''(k)\dz^2\big) v+\eps^p\pi_1(\bk)A_0 v\\
\label{eqrad3}
=
 \big[\pi_1(\bk)+\eps  \pi_1'(\bk)\cdot D-i\eps ({\bf b}\cdot \nabla)\pi_1(\bk)\big] F^{env}(\eps, \pi_1(\bk)v).
\end{eqnarray}
\end{example}

\subsection{Including ionization processes} \label{sectNLSioniz}

\subsubsection{The profile equation} 

As in \S~\ref{profileeq}, we solve the initial value problem 
\eqref{eqgen2}-\eqref{CI}-\eqref{CIVW} 
for times of order $O(1/\eps)$, writing the solution 
under a profile form,
 
\begin{equation} \label{eq2bis}
(\bU,\bW)(t,x)=(U,W)
\left( t,x,\frac{\bk\cdot x-\omega t}{\eps} \right). 
\end{equation}
Again, $U(t,x,\theta)$ and $W(t,x,\theta)$ 
are periodic with respect to $\theta$, and we use any 
$\omega\in \R$. 
The action of the Fourier multiplier 
${\mathcal H} \left( \eps\frac{{\bf k}}{k^2}\cdot D \right)$ 
from \eqref{Max4} (and \eqref{Hilbert}) is transferred 
at the profile level into the operator ${\mathcal H}
\left( D_\theta+\eps\frac{{\bf k}}{k^2}\cdot D \right)$, 
with $k=|{\bf k}|$, 
$$
\left( {\mathcal H} 
\Big(\eps\frac{{\bf k}}{k^2}\cdot D \Big) \bU \right) 
(t,x) = \left( {\mathcal H}
\Big( D_\theta + \eps\frac{{\bf k}}{k^2}
\cdot D \Big) U \right) 
\left( t,x,\frac{\bk\cdot x-\omega t}{\eps} \right) , 
$$
where ${\mathcal H}
\left( D_\theta+\eps\frac{{\bf k}}{k^2}\cdot D \right)$ 
is the Fourier multiplier
$$
{\mathcal H} \Big( D_\theta+\eps\frac{{\bf k}}{k^2}
\cdot D \Big) \sum_{n\in\Z} u_n(x) e^{in\theta} = 
\sum_{n\in\Z} \left( {\mathcal H}
\Big(n+\eps\frac{{\bf k}}{k^2}\cdot D \Big) u_n \right)(x) 
e^{in\theta},
$$
which acts continuously on $H^k(\T,B)$, for any $k\in\N$.  
In order to get a solution to the original 
problem, it is sufficient that $(U,W)$ solves 
\begin{equation} \label{eq3bis}
\left\lbrace
\begin{array}{l}
\dsp \dt U + A(\partial)U 
+ \frac{i}{\eps}\cL(\omega D_\theta,\bk D_\theta)U 
+ \eps^{1+p} A_0 U = \vspace{1mm}\\ 
\qquad \qquad 
\eps F(\eps, U) 
- \eps  {\mathcal H}\big(D_\theta+\eps\frac{{\bf k}}{k^2}\cdot D\big) 
(W C_1^T C_1 U )
- \eps c \, C_1^T G(C_1 U,W), \vspace{1mm}\\
\dsp \dt W - \frac{i}{\eps} \omega D_\theta W
= \eps G(C_1U,W) \cdot C_1 U , 
\end{array}\right.
\end{equation}
with initial conditions
\begin{equation}\label{ICP}
\dsp (U,W)_\init(x,\theta)=
\left(u^0(x)e^{i\theta}+\cc,0\right).
\end{equation}
\begin{theorem} \label{th1bis}
Let $B=A^n\times A$, with $A=H^{t_0}(\R^d)$ or $A=W(\R^d)$, 
and $u^0 \in A^n$. 
Under Assumptions \ref{assu1}, \ref{assu2} and \ref{assuF}, 
there exists $T>0$ such that for all $0<\eps\leq1$ 
there exists a unique solution 
$\bZ=(\bU,\bW)\in C\left([0,T/\eps];B\right)$ 
to \eqref{eqgen2}-\eqref{CI}-\eqref{CIVW}. 
Moreover, one can write $\bZ$ under the form
$$
\bZ(t,x)=Z\left(t,x,\frac{\bk\cdot x-\omega t}{\eps}\right),
$$
where $Z=(U,W)$ solves the \emph{profile equation} (\ref{eq3bis}), 
with the initial condition \eqref{ICP}.
\end{theorem}

\begin{proof}
Similar to the one of Theorem~\ref{th1},  
by an iterative scheme in $H^{k}(\T;B)$, 
with $k\geq1$.  
\end{proof}

\subsubsection{The slowly varying envelope approximation} 
In this section, as in in section~\ref{sectSVEA}, we shall 
assume that $\omega=\omega_1(\bk)$ is some characteristic 
frequency for $\cL(\cdot,\bk)$. Postulating the Ansatz 
\begin{equation}\label{apenvbis}
Z(t,x,\theta) \sim 
\left( u_{env}(t,x)e^{i\theta}+\cc , w_{env}(t,x) \right),
\end{equation}
we obtain formally the following system 
for $(u_{env},w_{env})$ (denoted $(u,w)$), 
\begin{equation}
\label{eq18bis}
\left\lbrace
\begin{array}{l}
\dsp \dt u + \frac{i}{\eps}\cL(\omega,\bk+\eps D) u
+ \eps^{1+p} A_0 u = \vspace{1mm}\\ 
\qquad \qquad \qquad \eps F^{env}(\eps, u) 
- \eps \, i w \, C_1^T C_1 u 
- \eps c \, C_1^T G^{env}(C_1 u,w) , \vspace{1mm} \\
\dsp \dt w = 2\eps G^{env}(C_1 u,w) \cdot \overline{C_1 u}, 
\end{array}\right.
\end{equation}
where we used that 
${\mathcal H} \Big( 1+\eps\frac{{\bf k}}{k^2} \cdot D \Big) 
= i + O(\eps)$. 
Here, $F^{env}$ is given by \eqref{defFenv} and $G^{env}$ 
is defined in the same way, filtering oscillations, 
\begin{align*}
G^{env}(u,w)&=\frac{1}{2\pi}
\int_0^{2\pi}e^{-i\theta}G(u e^{i\theta}+\cc,w) \, {\rm d}\theta
\end{align*}
\begin{remark} \label{rmkNLty}
The equation one obtains for $w$ from direct computations is actually
\begin{align*}
\dt w &= \frac{\eps}{2\pi}
\int_0^{2\pi}G(C_1 u e^{i\theta}+\cc,w) 
\cdot (C_1 u e^{i\theta}+\cc) \, {\rm d}\theta \\
& = c_1 \eps \Big( 2^K |C_1 u|^{2K} + 
2 \sum_{k=1}^{\lfloor K/2 \rfloor} \begin{pmatrix} K-1 \\ k \end{pmatrix} 
\begin{pmatrix} K-k \\ k \end{pmatrix} (2|C_1 u|^2)^{K-2k} |C_1 u \cdot C_1 u|^{2k} \Big) \\
& \hspace{9cm} + 2 c_2 \eps w |C_1 u|^2 \\
& = 2 \eps G^{env}(C_1 u,w) \cdot \overline{C_1 u},
\end{align*}
in view of Assumption~\ref{assu2}. 
\end{remark}

The approximation \eqref{apenvbis} is justified in the following theorem.
\begin{theorem}\label{th3bis}
Let $B=A^n\times A$, with $A=H^{t_0}(\R^d)$ or $A=W(\R^d)$. 
Let $u^0\in A^n$, with $\nabla u^0\in A^{nd}$,
and $r \in A^n$. 
Let Assumptions \ref{assu1}, \ref{assuF} and \ref{assu3} 
be satisfied and assume moreover that $\omega\neq0$ and that 
$u^0=\pi_1(\bk)u^0+\eps r$. Then
\item[(i)] There exist $T>0$ and, for all $\eps\in(0,1]$, 
a unique solution $(u,w) \in C([0,T/\eps];B^{(1)})$ 
to \eqref{eq18bis} with initial condition $(u^0,0)$.
\item[(ii)] There exists $\eps_0>0$ such that 
for all $0<\eps<\eps_0$, the solution $\bZ$ to \eqref{eq3bis} 
provided by Theorem \ref{th1bis} exists on $[0,T/\eps]$ and
$$
\abs{\bZ-\bZ_{SVEA}}_{L^\infty([0,T/\eps]\times\R^d)}
\leq \eps C(T,\abs{u^0}_{A})(1+\abs{\nabla u^0}_{A} 
+ \abs{r}_{A}),
$$
where $\dsp \bZ_{SVEA}(t,x) = 
\left(u(t,x) e^{i\frac{\bk\cdot x-\omega t}{\eps}}+\cc, 
w(t,x)\right)$.
\end{theorem}
\begin{proof}
The arguments are similar to the ones in the proof 
of Theorem~\ref{th3}. \\
\emph{Step 1.} As in the proof of Theorem~\ref{th1} 
and Theorem~\ref{th1bis}, we have local in time $T/\eps$ 
(with $T$ independent of $\eps$) existence of $(u,w)$, 
solution to \eqref{eq18bis},
together with (uniform w.r.t. $\eps$) bounds. \\
\emph{Step 2.} Decomposing $u=u_1+u_{II}$ as in Step 2 of 
the proof of Theorem~\ref{th3}, one obtains in the same way 
that $\abs{\dt u_1}_{L^\infty([0,T/\eps],A^n)}$ is bounded 
w.r.t. $\eps$. This is also the case for 
$\abs{\dt w}_{L^\infty([0,T/\eps],A)}$ (even of order $O(\eps)$), 
as shows directly the third equation in \eqref{eq18bis}. \\
\emph{Step 3.} As in Step 4 of the proof of Theorem~\ref{th3},
we deduce that 
$$
\dsp \frac{1}{\eps} \abs{u_{II}}_{L^\infty([0,T/\eps],A)} \leq 
C \left( T,\abs{u^0}_{A} \right) 
\left( 1 + \abs{\nabla u^0}_{A} + \abs{r}_{A} \right). 
$$
This is obtained by integration by parts 
in the integral formulation giving the $u_j's$ of $u_{II}$, 
using Step 2 and Assumption~\ref{assu3} to have non-stationary 
phase; we conclude by a Gronwall argument. \\
\emph{Step 4.} Approximation of $\bZ$ by $\bZ_{SVEA}$. 
Compared to Theorem \ref{th3}, the new point 
is the component $w$; for the sake of simplicity, we therefore set
$F=0$ throughout this proof. 
Denote 
$$\dsp Z_{\rm app}(t,x,\theta) =(U_{\rm app},W_{\rm app})(t,x,\theta) := 
\left(u(t,x) e^{i\theta}+\cc, w(t,x)\right), \quad 
\eps \widetilde Z = Z-Z_{\rm app},
$$
where $Z=(U,W)$ is the solution to the profile equation \eqref{eq3bis} 
provided by Theorem~\ref{th1bis}. 
We estimate $\widetilde Z = (\widetilde U,\widetilde W)$ 
in $H^{k}(\T;B)$ (defined in \eqref{defHk}), $k\geq1$. \\
$\bullet$ For $\widetilde W$. We have $\widetilde W_\init=0$ 
and 
\begin{equation*}
\begin{split}
\dt \widetilde W + \frac{i}{\eps} \omega D_\theta \widetilde W 
& = G(C_1 U,W) \cdot {C_1 U} 
- 2 G^{env}(C_1 u,w) \cdot \overline{C_1 u}  \\
& = c_1 \sum_{k\neq0} 
\left( \frac{1}{2\pi} \int_0^{2\pi} e^{-i\theta} G(C_1 U,W) \cdot C_1 U {\rm d}\theta \right) 
e^{ik\theta} \\
& \quad + c_2 W |C_1 U|^2  - 2 c_2 w |C_1 u|^2 .
\end{split}
\end{equation*}
The terms in $c_1$ and $c_2$ can be treated similarly. For the sake of
clarity, we therefore set $c_1=0$ and $c_2=1$ in this proof, so that the
right-hand side is given by
\begin{equation*}
W |C_1 U|^2 - W_{\rm app} |C_1 U_{\rm app}|^2  + 2 w\left( (C_1 u) \cdot (C_1 u) e^{2i\theta}+\cc \right). 
\end{equation*}
Since
$$
\big\vert{W |C_1 U|^2 -  W_{\rm app} |C_1 U_{\rm app}|^2
}\big\vert\leq \eps C(T,\abs{u^0}_A) \abs{\widetilde{Z}},
$$
we easily deduce that
\begin{eqnarray*}
\abs{\widetilde W(t)}_{H^{k}(\T;A)} \leq \eps C(T,\abs{u^0}_A)\int_0^t
\abs{\widetilde Z(t')}_{H^{k}(\T;A)} {\rm d}t'  
+ 2\Big\vert \int_0^t e^{-2it'\frac{\omega}{\eps}}
w (C_1 u) \cdot (C_1 u) \, {\rm d}t' \Big\vert_A .
\end{eqnarray*}
Splitting $u=u_1+u_{II}$ as in Theorem \ref{th3}, we have a uniform (in $\eps$) bound 
for $\dsp \frac{1}{\eps} 
\abs{w (C_1 u) \cdot (C_1 u_{II})}_{L^\infty([0,T/\eps],A^n)}$
by Step 1 and Step 3. The only component left to control is therefore
the one involving the product $(C_1 u_1) \cdot (C_1 u_1)$, for which we write
\begin{align*}
\int_0^t e^{-2it'\frac{\omega}{\eps}}
w (C_1 u_1) \cdot (C_1 u_1)
=&-i\frac{\eps}{2\omega}\int_0^t e^{-2it'\frac{\omega}{\eps}}\dt\big[w
(C_1 u_1) \cdot (C_1 u_1)\big] {\rm d}t' \\
&+i\frac{\eps}{2\omega} e^{-2it\frac{\omega}{\eps}}\big[w (C_1 u_1) \cdot (C_1 u_1)\big](t).
\end{align*}
Using the equation satisfied by $w$ to control $\dt w$ and Step 2 to
control $\dt u_1$, one readily deduces that
\begin{equation} \label{estimWtilde}
\abs{\widetilde W(t)}_{H^{k}(\T;A)} \leq 
C\left(T,\abs{u^0}_{A}\right) 
\left( 1+\abs{\nabla u^0}_{A} 
+ \abs{r}_{A} 
+
\eps\int_0^t \abs{\widetilde Z(t')}_{H^{k}(\T;A)} 
{\rm d}t'\right).
\end{equation}
$\bullet$ For $\widetilde U$. We have $\widetilde U_\init=0$ 
and  
\begin{eqnarray*}
\nonumber
\dsp \dt \widetilde U 
+ \frac{i}{\eps}\cL(\omega D_\theta,\bk D_\theta+\eps D) \widetilde U
+ \eps^{1+p} A_0 \widetilde U = 
\dsp \hfill 
- \Big( {\mathcal H} 
\big( D_\theta+\eps\frac{{\bf k}}{k^2}\cdot D \big) - i \Big) 
(W C_1^T C_1 U) \\
\dsp \hfill 
- i \, (W_{\rm app}+\eps\widetilde W) C_1^T C_1 (U_{\rm app}+\eps\widetilde U) 
+ i \, W_{\rm app} C_1^T C_1 U_{\rm app},
\end{eqnarray*}
where, for the sake of simplicity, we also have taken $c=0$ because
the corresponding terms do not raise any difficulty.
Since the 
Fourier multiplier 
${\mathcal H} \Big( D_\theta+\eps\frac{{\bf k}}{k^2}\cdot D \Big) - i$ 
has a norm of order $\eps$ when acting on $A$ and because the
difference of the last two terms in easily bounded in terms of $\eps\tilde
Z$, we get that $\abs{\widetilde U(t)}_{H^{k}(\T;A)}$ satisfies the
same upper bound as $\abs{\widetilde W(t)}_{H^{k}(\T;A)}$ in \eqref{estimWtilde}.
\medbreak

Gathering the upper bounds for $\widetilde{W}$ and $\widetilde{U}$ we therefore get
\begin{align*}
\sup_{t'\in[0,t]}  \abs{\widetilde Z(t')}_{H^{k}(\T;A)} 
& \leq 
C \left(T,\abs{u^0}_{A}\right) 
\left( 1+\abs{\nabla u^0}_{A} 
+ \abs{r}_{A} \right) \\
&
+ \, \eps \, C\left(T,\abs{u^0}_{A}\right) \int_0^t 
\sup_{t''\in[0,t']} \abs{\widetilde Z(t'')}_{H^{k}(\T;A)} 
{\rm d}t' .
\end{align*} 
By a Gronwall estimate, we finally obtain 
a bound on $\widetilde Z$, 
$$
\sup_{t\in[0,T/\eps]} \abs{\widetilde Z(t)}_{H^{k}(\T;B)}
\leq C \left( T,\abs{u^0}_{A} \right) 
\left( 1+\abs{\nabla u^0}_{A}+\abs{r}_{A} \right) , 
$$
which, since the $A-$ norm controls the $L^\infty-$ norm, 
immediately leads to the desired estimate on 
$\abs{\bZ-\bZ_{SVEA}}_{L^\infty([0,T/\eps]\times\R^d)}$. 
\end{proof}

\subsubsection{The nonlinear Schr\"odinger equation with ionization} 

As in \S~\ref{sectAs32}, a Schr\"odinger type equation can be derived
in presence of ionization, 
\begin{equation}\label{NLSpbis}
\left\lbrace\begin{array}{l}
\dsp \partial_t u + \cg\cdot\nabla u 
- \eps \frac{i}{2}\nabla\cdot H_\bk\nabla u 
+ \eps^{1+p} \pi_1(\bk) A_0 u = \\
\qquad \qquad \qquad \quad 
\eps \pi_1(\bk) \left( F^{env}(\eps, u) - i w C_1^T C_1 u 
- c \, C_1^T G^{env}(C_1 u,w) \right) , \vspace{1mm} \\
\dsp \dt w = 2 \eps G^{env}(C_1 u,w) \cdot \overline{C_1 u} .
\end{array}\right.
\end{equation}
Using Theorem \ref{th3bis}, this model can be justified with a
straightforward adaptation of the case without ionization.
\begin{coro}[Schr\"odinger approximation]\label{coroschrodbis}
Under the assumptions of Theorem \ref{th3bis}, one has for all 
$u^0 \in A^{(3)}$ such that $u_0=\pi_1(\bk)u_0$:
\item[(i)] There exists $T>0$ and, for all $\eps\in(0,1]$, 
a unique solution $(u,w) \in C([0,T/\eps];B^{(3)})$ 
to \eqref{NLSpbis} with initial condition $(u^0,0)$. 
\item[(2)]  There exists $\eps_0>0$ and ${\mathfrak c}_{NLS}>0$ 
such that for all $0<\eps<\eps_0$, the solution $\bZ$ to 
\eqref{eq3bis} provided by Theorem \ref{th1bis} 
exists on $[0,T/\eps]$ and
$$
\abs{\bZ-\bZ_{NLS}}_{L^\infty([0,T/\eps]\times\R^d)}
\leq \eps C \left( T,\abs{u^0}_A \right) 
\left(1+\abs{\nabla u^0}_A 
+ {\mathfrak c}_{NLS}\abs{u^0}_{A^{(3)}} \right),
$$
where $\bZ_{NLS}(t,x) = 
\left( u(t,x)e^{i\frac{\bk\cdot x-\omega t}{\eps}}+\cc,w(t,x) \right)$.
\end{coro}

\subsubsection{The most general model}

In \eqref{NLSpbis}, ionization effects have been added to the standard cubic
NLS equation. It is of course possible to add them to a most
sophisticated model that takes into account more general
nonlinearities, improved frequency dispersion and frequency dependent
polarization. We then obtain the following generalization of
\eqref{NLSpbis} (for which the same justification as in Corollary
\ref{coroschrodbis} holds),
\begin{equation}
	\label{eqpolarioniz}
	\left\lbrace
	\begin{array}{l}
	\dsp \big(1-i\eps{\bf b}\cdot\nabla-\eps^2\nabla\cdot B\nabla\big)
	\dt u\\
	\dsp\indent
	+\cg\cdot\nabla u-\eps\frac{i}{2} \nabla\cdot \big( H_\bk  +2 
        \cg\otimes{\bf b} \big)\nabla u+\eps^2 C_3(\nabla) u
	+ \eps^{1+p} \pi_1(\bk) A_0 u\\
        \dsp \indent\indent =\eps  \big[\pi_1(\bk)+\eps
        \pi_1(\bk)\pi_1'(\bk)\cdot D-i\eps ({\bf b}\cdot
        \nabla)\pi_1(\bk)\big] \Big( F^{env}(\eps, \pi_1(\bk)u) \\
\dsp \hspace{5cm} - \left( i w C_1^T C_1 u 
+ c \, C_1^T G^{env}(C_1 u,w) \right) \Big) , \\
\dsp \dt w = 2 \eps G^{env}(C_1 u,w) \cdot \overline{C_1 u} .
	\end{array}\right.
\end{equation}
As shown in Appendix \ref{APPmg}, this system of equations takes the following
form corresponding to  (\ref{NLSfamily2}) in the case of Maxwell's equations,
\begin{equation} \label{NLSfamily2gen}
\left\lbrace
\begin{array}{l} 
\!\!\!\dsp i (P_2(\eps\nabla)\dt + c_g\dz ) u
+\eps (\Delta_\perp \!+\!\alpha_1\dz^2) u + i \eps \alpha_2 u \vspace{1mm} \\
\qquad \qquad + \eps (1 + i \eps {\bf c_3}\cdot\nabla) [(\abs{u}^2\! -\!\rho) u  
+ i c (\alpha_4 \abs{u}^{2K-2}\! u+\alpha_5\rho u)] = 0 , \vspace{1mm} \\
\!\!\!\dsp \dt \rho = \eps \alpha_4 \abs{u}^{2K}+\eps \alpha_5 \rho \abs{u}^2.
\end{array}\right.
\end{equation}
\section{Analysis of the  equations \eqref{NLSfamily1} 
and (\ref{NLSfamily2}), and open problems}\label{sec:openproblem}

In this section, we analyze and formulate open problems for the NLS-type equations (\ref{NLSfamily1}), as well as the NLS-type equations (\ref{NLSfamily2}) taking the ionization processes into account. We first consider in section \ref{sec:rugby1} the equations \eqref{NLSfamily1} in the case of respectively no and anomalous GVD (resp. $\alpha_1=0$ and $\alpha_1=1$), and also briefly discuss (\ref{NLSfamily2}). Then, we discuss \eqref{NLSfamily1} in the case of normal GVD (i.e. $\alpha_1=-1$) in section \ref{sec:rugby2}. Finally, we formulate additional open problems in section \ref{sec:rugby3}, section \ref{sec:rugby4} and section \ref{sec:rugby5}.

\subsection{The case of no and anomalous GVD (resp. $\alpha_1=0$ and $\alpha_1=1$)}\label{sec:rugby1}

Let us first consider the case where $P_2, \alpha_2,  \boldsymbol{\alpha}_3$ and $f$ take the following values,
$$P_2=I, \quad \alpha_2=0, \quad  \boldsymbol{\alpha}_3=0\quad \textrm{ and }\quad f=0,$$
in which case equation \eqref{NLSfamily1} corresponds to the focusing cubic NLS \eqref{eqNLS} in dimension $d=2$ in the case of no GVD, and in dimension $d=3$ in the case of anomalous GVD. As recalled in the introduction, \eqref{eqNLS} is critical in dimension 2, and supercritical in dimension 3, and there exist finite time blow-up solutions in both cases. Let us also recall that 
some of these blow-up dynamics are stable (see for example \cite{MR1} \cite{MR2} \cite{MR3} \cite{MR4} \cite{MR5} in the critical case, and \cite{MRS} in a slightly supercritical case).

Now, let us recall that the breakdown of solutions is not always
observed in physical experiments on the propagation of laser beams and
that instead a {\it filamentation} phenomenon occurs. In this section, we would like to analyze the possibility that the modified model \eqref{NLSfamily1} in the case of no or anomalous GVD prevents the formation of singularities. Below, we analyze the role of each parameter of \eqref{NLSfamily1} separately starting with the nonlinearity $f$.

\subsubsection{The nonlinearity}\label{sec:xoxoxoxolol}

In the case where $P_2=I, \alpha_2=0$, and $ \boldsymbol{\alpha}_3=0$, (\ref{NLSfamily1}) takes the following form,
\begin{equation}\label{rg1}
i\dt v + \Delta v +\big(1+f(\varepsilon^r \abs{v}^2)\big)\abs{v}^2v=0,\,\,\,t>0,\,x\in\R^d,
\end{equation}
where the dimension is $d=2$ in the case of no GVD, and $d=3$ in the case of anomalous GVD. 

As recalled in the introduction, standard modifications of the cubic nonlinearity consist either of the  cubic/quintic nonlinearity, i.e. $f(s)=-s$, or a saturated nonlinearity, i.e. $f$ is a smooth function on $\R^+$ vanishing at the origin and such that $(1+f(s))s$ is bounded on $\R^+$ (e.g. $f(s)=-\frac{s}{1+s}$). Let us first consider the case of a saturated nonlinearity. In that case, the fact that $(1+f(s))s$ is bounded on $\R^+$ implies the following control of the nonlinear term in \eqref{rg1}:
$$\|(1+f(\varepsilon^r \abs{v}^2))\abs{v}^2v\|_{L^2}\lesssim \frac{\|v\|_{L^2}}{\varepsilon^r}.$$
Thus, running a fixed point argument yields the fact that this equation is locally well posed in 
$L^2(\mathbb{R}^d)$ for any integer $d\geq 1$, with a time of existence only depending on the size of 
$\|u_0\|_{L^2}$. Since equation \eqref{rg1} still satisfies the conservation of mass, this immediately implies 
global existence for any initial data in $L^2(\R^d)$ and for any $d\geq 1$. Therefore, modifying the 
cubic nonlinearity in the standard NLS equation by a saturated nonlinearity does indeed prevent finite-time 
breakdown of the solutions.

Next, we consider the modification by a cubic/quintic nonlinearity. In that case, \eqref{rg1} becomes,
\begin{equation}\label{rg2}
i\dt v + \Delta v +\abs{v}^2v-\varepsilon^r \abs{v}^4v=0,\,\,\,t>0,\,x\in\R^d.
\end{equation}
Note that \eqref{rg2} still satisfies the conservation of mass, and the conserved energy is now given by,
$$E(v(t))=\frac{1}{2}\int|\nabla v(t,x)|^2{\rm d}x-\frac{1}{4}\int |v(t,x)|^{4}{\rm d}x+\frac{\varepsilon^r}{6}\int |v(t,x)|^6{\rm d}x=E(v_0).$$
Note also that,
\begin{eqnarray*}
E(v(t))+\left(\frac{1}{2}+\frac{3}{32\varepsilon^r}\right)\|v(t)\|^2_{L^2}&=&\frac{1}{2}\|v(t)\|^2_{H^1}+\frac{\varepsilon^r}{6}\int|v(t)|^2\left(|v(t)|^2-\frac{3}{4\varepsilon^r}\right)^2 {\rm d}x\\
&\geq &  \frac{1}{2}\|v(t)\|^2_{H^1}
\end{eqnarray*}
which together with the conservation of mass and energy yields,
\begin{equation}\label{rg3}
\|v(t)\|_{H^1}\leq \sqrt{2E(v_0)+\left(1+\frac{3}{16\varepsilon^r}\right)\|v_0\|^2_{L^2}}.
\end{equation}

Let us first consider the case of dimension 2. In this case, both the cubic 
focusing NLS and the quintic defocusing NLS are $H^1$ subcritical in 
the sense that the equation is locally well posed in $H^1(\R^2)$ with a 
time of existence only depending on the size of $\|u_0\|_{H^1}$ (see \cite{GV}). 
The proof extends to \eqref{rg2} which is thus locally well posed in 
$H^1(\R^2)$ with a time of existence only depending on the size of 
$\|u_0\|_{H^1}$. Together with the bound \eqref{rg3}, this immediately implies 
global existence for any initial data in $H^1(\R^2)$. The case of dimension 3 is more involved, 
since the quintic defocusing NLS is $H^1$ critical in the sense that the equation is locally well posed in 
$H^1(\R^3)$ with a time of existence depending on the shape of $u_0$ (see \cite{CW}). 
Thus, one cannot rely solely on the bound \eqref{rg3} to prove global existence. 
However, the case of combined nonlinearities 
is addressed in \cite{TVZ}, where global existence of solutions with $H^1$ initial data is proved, 
when the nonlinearity is a sum of two $H^1$-subcritical powers, or the sum of an $H^1$-subcritical power 
and of a defocusing $H^1$-critical power. This includes equation \eqref{rg2}. Therefore, modifying the 
cubic nonlinearity in the standard NLS equation by a quintic defocusing nonlinearity 
does indeed prevent finite-time 
breakdown of the solutions both in dimension 2 and 3. Let us mention an interesting phenomenon  
regarding the qualitative behavior of the solutions. Physical experiments 
suggest an oscillatory behavior of the solution which focuses, then  
defocuses, refocuses,... in an almost periodic fashion (see \cite{BergeSkupin} 
and references therein). Such a behavior is also observed in numerical simulations 
and suggested by heuristic arguments (see for example \cite{SS} and \cite{FP}).

\begin{openproblem} 
 Establish rigorously the ``oscillatory'' phenomenon of the solutions to \eqref{rg1} observed in physical experiments (see e.g. \cite{BergeSkupin}). 
\end{openproblem}

\noindent Another open problem concerns the behavior of the 
solution as $\varepsilon\rightarrow 0$. Pick an initial data 
$v_0\in H^1(\R^d)$, $d=2, 3$, leading to 
a finite time blow-up solution $v$ at time $T>0$ to the cubic NLS \eqref{eqNLS}. 
Consider the solution $v_\varepsilon$ of \eqref{rg1} with the same initial data $v_0$. 
In the case of dimension 2, it is shown in \cite{M1} that $v_\varepsilon(t)\rightarrow v(t)$ in $H^1(\R^2)$ as $\varepsilon\rightarrow 0$ on $[0,T)$ 
and,
$$\lim_{\varepsilon\rightarrow 0}\|v_\varepsilon(T)\|_{H^1(\R^2)}=+\infty.$$
An interesting question is the understanding of the limit of $v_\varepsilon(t)$ as 
$\varepsilon\rightarrow 0$ for $t>T$. Partial results have been obtained in this direction 
in \cite{M2}. There, it is proved that only few scenarios are possible, but one would like to establish whether 
all scenarios do occur or only some of them, and which scenarios are generic.

\begin{openproblem}\label{OP2}
Let $v_0\in H^1(\R^d)$, $d=2, 3$ be leading to a finite time blow-up solution $v$ at time $T>0$ to the cubic NLS \eqref{eqNLS}. Consider the solution $v_\varepsilon$ of \eqref{rg1} 
with the same initial data $v_0$. Describe the behavior of $v_\varepsilon(t)$ as 
$\varepsilon\rightarrow 0$ for $t>T$. 
\end{openproblem}

\subsubsection{Taking the ionization process into account}

Let us discuss equations (\ref{NLSfamily2}) and \eqref{NLSfamily3} 
in the case of no or anomalous GVD  which take the ionization process into account and also correspond 
to a modification of the nonlinearity in the focusing 
cubic NLS \eqref{eqNLS}. Ionization is certainly the most important
phenomenon leading to the formation of laser filaments. Recall that (\ref{NLSfamily2}) is given by
\begin{equation}\label{rg4}
\left\lbrace
\begin{array}{l}
\!\!\!\dsp i (\dt + c_g\dz ) u
+\eps (\Delta_\perp \!+\!\alpha_1\dz^2) u + \eps (\abs{u}^2\! -\!\rho) u  = 
\!- i\eps c(\alpha_4 \abs{u}^{2K-2}\! u+\alpha_5\rho u) , \vspace{1mm} \\
\!\!\!\dsp \dt \rho = \eps \alpha_4 \abs{u}^{2K}+\eps \alpha_5 \rho \abs{u}^2,
\end{array}\right.
\end{equation}
where $\alpha_4, \alpha_5\geq 0$, $c>0$, where $\rho$ is the density
of electron created by ionization, while $\cg=c_g {\bf e}_z$ is the
group velocity associated to the laser pulse, and  where $d=2$ in the case 
of no GVD, and $d=3$ in the case of anomalous GVD. Also, recall that (\ref{NLSfamily3}) 
is given by
\begin{equation}\label{rg4biteux}
\left\lbrace
\begin{array}{l}
\dsp i \dta v
+(\Delta_\perp +\alpha_1\dz^2) v +  (\abs{v}^2 -\tilde\rho) v  = 
- i c(\alpha_4 \abs{v}^{2K-2}v+\alpha_5\tilde\rho v) , \vspace{1mm} \\
\dsp -c_g\dz \tilde\rho = \eps \alpha_4 \abs{v}^{2K}+\eps \alpha_5 \tilde\rho \abs{v}^2,
\end{array}\right.
\end{equation}
where $\tilde\rho$ corresponds to $\rho$ written in a frame moving at the group velocity 
$\cg=c_g{\bf e}_z$ and with respect to a rescaled time $\tau=\eps t$, i.e.
$$\rho(t,X_\perp,z) =\tilde \rho(\eps t, X_\perp,z-c_g t).$$
In view of the second equation of \eqref{rg4biteux}, a boundary condition for $\tr$ has to be prescribed at $z=z_0$ for some $z_0$ in order to obtain a well-posed problem. A natural choice, which ensures that $\tr\geq 0$, is to prescribe $\tr$ at $+\infty$: 
\begin{equation}\label{boundarybiteux}
\lim_{z\rightarrow +\infty}(\nabla^\perp)^l\tr(\tau, X_\perp, z)=0,\,\,\forall\, l\geq 0.
\end{equation}

Let us first discuss the local well-posedness theory for equations \eqref{rg4} and \eqref{rg4biteux}, starting with the first one. One may obtain the existence of solutions to \eqref{rg4} over an interval of time with size $O(\eps^{-1})$. Indeed, for an integer $N>d/2$, differentiating $N$ times with respect to space variables the equations both for $v$ and $\rho$, multiplying \eqref{rg4} respectively by $\nabla^N\overline{v}$ and $\nabla^N\rho$, 
integrating on $\R^d$, taking the imaginary part and integrating by parts for the first equation, and summing both equations yields: 
\begin{eqnarray*}
\frac{d}{\rm dt}[\|v\|^2_{H^N}+\|\rho\|^2_{H^N}]&\lesssim& \eps(\|v\|^2_{H^N}+\|\rho\|^2_{H^N})(\|v\|_{L^\infty}+\|v\|^{2K-2} _{L^\infty}+\|\rho\|_{L^\infty})\\
&\lesssim& \eps(\|v\|^2_{H^N}+\|\rho\|^2_{H^N})(1+\|v\|^2_{H^N}+\|\rho\|^2_{H^N})^{K-1}
\end{eqnarray*}
where we used the fact that $N>d/2$ together with simple product rules and the Sobolev embedding. Integrating this differential inequality, we obtain a time of existence with size $O(\eps^{-1})$.

Next, we discuss the local well-posedness theory for equation \eqref{rg4biteux} supplemented with the boundary condition \eqref{boundarybiteux}. For any integer $N$, we introduce the space $\HH^N$, 
$$\HH^N=\left\{f : \R^d \rightarrow \R , \sum_{j=0}^N\|\nabla^jf\|_{L^2_{X_\perp}L^\infty_z}<+\infty\right\}.$$
For an integer $N>d$, differentiating $N$ times with respect to space variables the equations for $v$,  multiplying the first equation of \eqref{rg4biteux} by $\nabla^N\overline{v}$, integrating on $\R^d$, taking the imaginary part and integrating by parts yields: 
\begin{eqnarray*}
\frac{d}{\rm dt}\|v\|^2_{H^N}&\lesssim& \|v\|^4_{H^N}+\|v\|^{2K}_{H^N}+\|v\|^2_{H^N}\|\tr\|_{\HH^N}
\end{eqnarray*}
where we used the fact that $N>d$ together with the Sobolev embedding in the last inequality. Next, we estimate $\|\tr\|_{\HH^N}$. In view of the second equation of \eqref{rg4biteux} and the boundary condition \eqref{boundarybiteux}, we have
$$\tr(t,x,y,z)=\int_z^{+\infty}\alpha_4\eps |v|^{2K}(t,x,y,\sigma)\exp\left(\frac{\alpha_5\eps}{c}\int_z^{\sigma}|v|^2(t,x,y,s){\rm d}s\right){\rm d}\sigma.$$
We infer
$$|\tr(t,x,y,z)|\lesssim \left(\int_{-\infty}^{+\infty}|v|^{2K}(t,x,y,\sigma){\rm d}\sigma\right)\exp(\|v\|^2_{H^N}),$$
which yields
$$\|\tr\|_{\HH^0}\lesssim \|v\|^{2K}_{H^N}\exp(\|v\|^2_{H^N}),$$
where we used the fact that $N>d$ together with the Sobolev embedding. Next, differentiating the second equation of \eqref{rg4biteux} N times, we obtain
$$-c\partial_z\nabla^N\tr = \alpha_4 \varepsilon\nabla^N(\abs{v}^{2K})+\alpha_5 \varepsilon\nabla^N(\tr \abs{v}^2).$$
In view of the boundary condition \eqref{boundarybiteux}, this yields
\begin{eqnarray*}
\nabla^N\tr(t,x,y,z) &=& \int_z^{+\infty}\left(\alpha_4\eps\nabla^N(|v|^{2K})(t,x,y,\sigma)+\alpha_5\eps(\nabla^N(|v|^2\tr)-|v|^2\nabla^N\tr)(t,x,y,\sigma)\right)\\
&&\times\exp\left(\frac{\alpha_5\eps}{c}\int_z^{\sigma}|v|^2(t,x,y,s){\rm d}s\right){\rm d}\sigma.
\end{eqnarray*}
Hence, we deduce
\begin{eqnarray*}
&&|\nabla^N\tr(t,x,y,z)|\\
&\lesssim& \left(\int_{-\infty}^{+\infty}\left(|\nabla^N(|v|^{2K})(t,x,y,\sigma)+|\nabla^N(|v|^2\tr)-|v|^2\nabla^N\tr|(t,x,y,\sigma)\right){\rm d}\sigma\right)\exp(\|v\|^2_{H^N}) ,
\end{eqnarray*}
which yields
$$\|\tr\|_{\HH^N}\lesssim (\|\tr\|_{\HH^{N-1}}\|v\|^2_{H^N}+\|v\|^{2K}_{H^N})\exp(\|v\|^2_{H^N}).$$
By induction, we obtain
$$\|\tr\|_{\HH^N}\lesssim (1+\|v\|^{2N}_{H^N})\|v\|^{2K}_{H^N}\exp((N+1)\|v\|^2_{H^N}).$$
This yields
\begin{eqnarray*}
\frac{d}{\rm dt}\|v\|^2_{H^N}&\lesssim& \|v\|^4_{H^N}+\|v\|^{2K}_{H^N}+\|v\|^2_{H^N}(1+\|v\|^{2N}_{H^N})\|v\|^{2K}_{H^N}\exp((N+1)\|v\|^2_{H^N}).
\end{eqnarray*}
Integrating this differential inequality, we obtain local existence for $(v, \tr)$ in $H^N\times\HH^N$ for $N>d$.

\begin{remark}
Note that the factor $\eps$ in the second equation of \eqref{rg4biteux}, while present in the model, is not needed for the well-posedness theory. 
\end{remark}

Let us now come back to the issue of global well-posedness/finite time singularity formation. For equations  
\eqref{rg4} and \eqref{rg4biteux}, note that the $L^2$ norm of $u$ (resp. $v$) is dissipated. Indeed, in the case of 
\eqref{rg4biteux}, multiplying by $v$, 
integrating on $\R^d$, taking the imaginary part and integrating by parts 
yields: 
$$\frac{d}{\rm dt}\|v\|^2_{L^2}= -c\left(\alpha_4\|v\|^{2K}_{L^{2K}}+\alpha_5\int |v|^2\tr {\rm d}x\right),$$
from which we deduce that the $L^2$ norm is dissipated since $\tr\geq 0$. On the other hand, the energy 
is not conserved, nor dissipated. We thus can not carry out the analysis of section \ref{sec:xoxoxoxolol} 
 for equation \eqref{rg2}, even in the case of dimension 2\footnote{Note that a simpler model of damped NLS 
 where the $\rho$ or $\tr$ term are not present has been investigated in \cite{AnCaSp}. The authors obtain 
 global existence for a certain range of parameters by controlling a modified energy even if it is not conserved.}. 
 An interesting problem would then be to prove 
that taking the ionization process into account 
(i.e. replacing the cubic focusing NLS \eqref{eqNLS} by either 
equation \eqref{rg4} or equation \eqref{rg4biteux}) does indeed prevent finite-time breakdown 
of the solutions in dimensions 2 and 3. This is formulated in 
the following open problem. 

\begin{openproblem}
Prove that the solutions to equations \eqref{rg4} and \eqref{rg4biteux} are global in dimensions 2 and 3. 
\end{openproblem}

\subsubsection{The damping}

In the case where $P_2=I,  \boldsymbol{\alpha}_3=0$ and $f=0$, (\ref{NLSfamily1}) takes the following form,
\begin{equation}\label{rg5}
i \dt v + \Delta v +i \alpha_2 v +\abs{v}^2 v=0,
\end{equation}
where $d=2$ in the case of no GVD, and $d=3$ in the case of anomalous GVD. 
The mass and energy are not conserved quantities anymore. The mass decreases,
\begin{equation}\label{rg6}
\|v(t)\|_{L^2}=e^{-\alpha_2t}\|v_0\|_{L^2},
\end{equation}
while for the energy we have,
\begin{equation}\label{rg7}
\frac{d}{\rm dt}\left[\frac{1}{2}\|\nabla v\|^2_{L^2}-\frac{1}{4}\int |v|^4 {\rm d}x\right]=
-\alpha_2\left(\|\nabla v\|^2_{L^2}-\int |v|^4 {\rm d}x \right),
\end{equation}
so that the energy is neither increasing nor decreasing. Equation \eqref{rg5} has been analyzed in several works (see e.g. \cite{T}, \cite{SS}, \cite{F} and \cite{OT}). In particular, from standard arguments using Strichartz estimates, there is in dimensions 2 and 3 a continuous and increasing function $\theta$ such that global existence holds for $\alpha_2>\theta(\|u_0\|_{H^1})$ (see \cite{SS} and \cite{OT}). Also, in the 3-dimensional case, there is a function $\theta$ such that stable blow-up solutions exist for $0\leq \alpha_2\leq \theta(u_0)$ (see \cite{T} and \cite{OT}). Finally, the same behavior has been  observed in numerical simulations and suggested by heuristic arguments in the 2-dimensional case (see \cite{F}). Thus, a small $\alpha_2$ does not prevent finite time blow up. Since 
the constant $\alpha_2$ obtained in \eqref{NLSfamily2} is usually small, it appears that the damping cannot by itself explain the physical observation according to which the breakdown of solutions does not occur.

\subsubsection{Off-axis variation of the group velocity}

In the case where $\alpha_2=0,  \boldsymbol{\alpha}_3=0$ and $f=0$, (\ref{NLSfamily1}) takes the following form,
\begin{equation}\label{rg8}
i P_2(\varepsilon \nabla)\dt v + \Delta v  +\abs{v}^2v=0,
\end{equation}
where $d=2$ in the case of no GVD, and $d=3$ in the case of anomalous GVD. The energy remains unchanged and is still conserved, while the mass is replaced by the following conserved quantity,
\begin{equation}\label{rg9}
(P_2(\varepsilon\nabla)v(t),v(t))=(P_2(\varepsilon\nabla)v_0,v_0).
\end{equation}

Let us first consider the case of full off-axis dependence, i.e. the case where,
\begin{equation}\label{rg10}
(P_2(\varepsilon\nabla)v,v)\gtrsim \|v\|^2_{L^2}+\varepsilon^2 \|\nabla v\|^2_{L^2}.
\end{equation}
The operator $P_2(\varepsilon \nabla)$ is a second order self-adjoint operator which is invertible in view of \eqref{rg10}. We denote by $P_2(\varepsilon \nabla)^{-1}$ its inverse. We rewrite \eqref{rg8} in the following form,
\begin{equation}\label{zob}
i \dt v + P_2(\varepsilon \nabla)^{-1}\Delta v  +P_2(\varepsilon \nabla)^{-1}(\abs{v}^2v)=0.
\end{equation}
Using Duhamel's formula, and the semi-group $e^{itP_2(\varepsilon \nabla)^{-1}\Delta}$, we obtain,
\begin{equation}\label{zob1}
v=e^{itP_2(\varepsilon \nabla)^{-1}\Delta}v_0+\int_0^t e^{i(t-s)P_2(\varepsilon \nabla)^{-1}\Delta}P_2(\varepsilon \nabla)^{-1}(\abs{v}^2v)(s){\rm d}s.
\end{equation}
As a simple consequence of \eqref{rg10} and the Sobolev embedding in dimensions 2 and 3, we have,
\begin{equation}\label{zob2}
\|P_2(\varepsilon \nabla)^{-1}(\abs{v}^2v)\|_{H^1(\R^d)}\lesssim \frac{1}{\varepsilon}\|\abs{v}^2v\|_{L^2(\R^d)}\lesssim \frac{1}{\varepsilon}\|v\|_{H^1(\R^d)}^3,\,\,\, d=2, 3.
\end{equation}
Now, a fixed point argument based on the formulation \eqref{zob1} together with the estimate \eqref{zob2} and the fact that the semi-group $e^{itP_2(\varepsilon \nabla)^{-1}\Delta}$ is unitary on $H^1$, 
implies that \eqref{rg8} is locally well posed in $H^1$ for dimensions $d=2, 3$, with a time of 
existence only depending on the size of $\|u_0\|_{H^1}$. Together with the lower bound 
\eqref{rg10} and the conserved quantity \eqref{rg9} which yield a uniform in time bound 
on the $H^1$ norm of the solution, this immediately implies global existence for any 
initial data in $H^1$ when $d=2, 3$. Therefore, modifying the focusing cubic NLS equation 
by adding a full off-axis dependence does indeed prevent the finite-time breakdown of the solutions. 

Next, let us consider the case of partial off-axis dependence, i.e. the case where,
\begin{equation}\label{rg11}
(P_2(\varepsilon\nabla)v,v)\gtrsim \|v\|^2_{L^2}+\varepsilon^2 \sum_{k=1}^j\|v_k\cdot\nabla v\|^2_{L^2},
\end{equation}
for vectors $v_k$ in $\R^d$, $k=1,\ldots, j$, where $j=1$ if $d=2$, and $j=1$ or $j=2$ if $d=3$ so that the right-hand side of \eqref{rg11} does not control the full $H^1$ norm. Notice that the well-posedness theory has yet to be investigated in this case. Indeed, consider the nonlinear term in the formulation  \eqref{zob1},
\begin{equation}\label{zob3}
e^{it P_2(\varepsilon \nabla)^{-1}\Delta}P_2(\varepsilon \nabla)^{-1}(\abs{v}^2v).
\end{equation}
In the directions $v_k$ of \eqref{rg11}, the operator $P_2(\varepsilon \nabla)^{-1}$ gains two derivatives, while the semi-group $e^{it P_2(\varepsilon \nabla)^{-1}\Delta}$ does not disperse (and thus does not satisfy a useful Strichartz estimate). On the other hand, in the directions orthogonal to the vectors $v_k$,  the operator $P_2(\varepsilon \nabla)^{-1}$ does not gain any derivative, while the semi-group $e^{it P_2(\varepsilon \nabla)^{-1}\Delta}$ should satisfy a useful Strichartz estimate. Thus, in order to investigate the well-posedness theory, one should try to combine the regularization provided by the operator $P_2(\varepsilon \nabla)^{-1}$ in the directions $v_k$ of \eqref{rg11} with the dispersive properties of the semi-group $e^{it P_2(\varepsilon \nabla)^{-1}\Delta}$ in the direction orthogonal to the vectors $v_k$. Now, it would be interesting to investigate whether a suitable well-posedness theory together with the conserved quantities given by the energy and \eqref{rg9} yield global existence. 
This suggests the following open problem. 

\begin{openproblem}
Investigate both the local and global well-posedness for equation \eqref{rg8} in the case of partial off-axis dependence. In particular,  does the modification of the focusing cubic NLS \eqref{eqNLS} by the addition of a partial off-axis dependence prevent the finite-time breakdown of the solutions?
\end{openproblem}

\subsubsection{Self-steepening of the pulse}

In the case where $P_2=I, \alpha_2=0$ and $f=0$, (\ref{NLSfamily1}) takes the following form,
\begin{equation}\label{rg12}
i \dt v + \Delta v +(1+i \varepsilon  \boldsymbol{\alpha}_3\cdot \nabla )\abs{v}^2 v=0,
\end{equation}
where $d=2$ in the case of no GVD, and $d=3$ in the case of anomalous GVD. Adding 
the operator $(1+i \varepsilon  \boldsymbol{\alpha}_3\cdot \nabla )$ in front of the nonlinearity of the focusing 
cubic NLS does certainly not prevent finite-time breakdown of the solutions. Indeed, one 
expects to obtain even more blow up solutions since the formation of optical shocks is expected 
in this case. The term shock is used in view of the similarity with the Burgers equation - both in terms of the nonlinearity of the equation and the profiles of the solutions observed in some numerical simulations (see \cite{AL}  for more details on optical shocks).  
The reason we included this modification in the discussion is because it 
may account for the physical observation of the self-steepening of the pulse: an initial pulse which 
is symmetric becomes asymmetric after propagating over a large distance and its profile seems 
to develop a shock. This phenomenon has been widely observed and we refer the reader to \cite{BergeSkupin} and references therein. Thus, an interesting open problem would be to exhibit solutions of \eqref{rg12} for which the corresponding profile develops a shock. We formulate a slightly more general open problem below.

\begin{openproblem}
Describe the blow-up scenarios for the finite time blow-up solutions to equation \eqref{rg12}.
\end{openproblem}

\subsection{The case of normal GVD (i.e. $\alpha_1=-1$)}\label{sec:rugby2}

Let us consider the case where $P_2, \alpha_1, \alpha_2,  \boldsymbol{\alpha}_3$ and $f$ take the following value,
$$P_2=I, \alpha_1=-1, \alpha_2=0,  \boldsymbol{\alpha}_3=0\textrm{ and }f=0 ,$$
in which case equation \eqref{NLSfamily1} corresponds to the following equation in dimension 2 or 3,
\begin{equation}\label{rg13}
i\dt v + (\Delta_\perp -\dz^2) v +\abs{v}^2 v =0.
\end{equation}
This equation is sometime referred to as hyperbolic cubic NLS, or non elliptic cubic NLS. Since Strichartz estimates only depend on the curvature of the corresponding characteristic manifold (see \cite{Str}), equation \eqref{rg13} satisfies the same Strichartz estimates as the cubic focusing NLS. Thus, the result of Ginibre and Velo \cite{GV} immediately extends to \eqref{rg13} which is  locally well-posed in $H^1=H^1(\R^d)$ with $d=2, 3$. Therefore, for $v_0\in H^1$, there exists $0<T\leq +\infty$ and a unique solution $v(t)\in {\mathcal{C}}([0,T),H^1)$ to (\ref{rg13}) and either $T=+\infty$, and the solution is global, or the solution blows up in finite time $T<+\infty$ and then $\lim_{t\uparrow T}\|\nabla u(t)\|_{L^2}=+\infty$. 

Equation (\ref{rg13}) admits the following conservation laws in the energy space $H^1$,
$$
\begin{array}{lll}
L^2-\mbox{norm}: \ \ \|v(t)\|^2_{L^2}=\|v_0\|_{L^2}^2; \vspace{1mm} \\
\displaystyle 
\mbox{Energy}:\ \ E(v(t))=\frac{1}{2}\int|\nabla_\perp v(t,x)|^2{\rm d}x
-\frac{1}{2}\int |\partial_zv|^2{\rm d}x 
-\frac{1}{4}\int |v(t,x)|^{4}{\rm d}x=E(v_0); \vspace{1mm} \\
\displaystyle 
\mbox{Momentum}:\ \ Im \left( \int\nabla v(t,x)\overline{v(t,x)}{\rm d}x \right) =
Im \left( \int\nabla v_0(x)\overline{v_0(x)}{\rm d}x \right).
\end{array}
$$
A large group of $H^1$ symmetries leaves the equation invariant: if $u(t,x)$ solves \eqref{rg13}, then $\forall (\lambda_0,\tau_0,x_0,\gamma_0)\in \R_*^+\times\R\times\R^3\times \R$, so does 
\begin{equation}
\label{rg14}
u(t,x)=\lambda_0v(\lambda_0^2t+\tau_0,\lambda_0 x+x_0)e^{i\gamma_0}.
\end{equation}
Note that \eqref{rg13} is not invariant under the usual Galilean transform. However, it is invariant under a twisted Galilean transform. For $\beta=(\beta_1, \beta_2, \beta_3)\in\R^3$, we define,
$$\hat{\beta}=(\beta_1, \beta_2, -\beta_3)\in\R^3.$$
Then \eqref{rg13} is invariant under the following twisted Galilean transform,
\begin{equation}\label{rg15}
v_\beta(t,x)=v(t, x-\beta t)e^{i\frac{\hat{\beta}}{2}\cdot(x-\frac{\beta}{2}t)}.
\end{equation}
Note also that the scaling symmetry $u(t,x)=\lambda_0v(\lambda^2_0t,\lambda_0 x)$ leaves the space $L^2(\R^2)$ invariant so that \eqref{rg13} is critical with respect to the conservation of mass in dimension 2, while it leaves the homogeneous Sobolev space $\dot{H}^{\frac{1}{2}}(\R^3)$ invariant so that \eqref{rg13} is supercritical with respect to the conservation of mass in dimension 3. 

In contrast to the cubic focusing NLS, the existence or absence of finite time blow up solutions for \eqref{rg13} is widely open. While there exists a counterpart to the virial for \eqref{rg13}, it is too weak to provide the existence of finite-time blow up dynamics (see the discussion in \cite{SS}). Also, 
there are no standing waves in the form $v(t,x)=Q(x)e^{i\omega t}, \omega\in \R$ in the energy space for the equation \eqref{rg13} (see \cite{GS}), while for the focusing cubic NLS, such an object exists and is fundamental in the analysis of the blow up dynamics (see e.g. \cite{MR1} \cite{MR2} \cite{MR3} \cite{MR4} \cite{MR5} \cite{MR6} in dimension 2, and \cite{MRS} in dimension 3).
 Now, numerical simulations suggest that the solutions to \eqref{rg13} do not break down in finite time (see e.g. \cite{SS}, \cite{FP} and references therein). In particular, the simulations exhibit a phenomenon called pulse splitting where the pulse focuses until it reaches a certain threshold. Once this threshold is attained, the pulse splits in two pulses of less amplitude moving away from each other. This phenomenon might repeat itself (multiple splitting) and ultimately prevent finite time blow up, but this has not been clearly backed up by numerics so far. Thus an interesting problem would be to prove that the cubic NLS equation \eqref{rg13} (i.e. in the case of normal GVD) does not have finite time blow up solutions, whether this is due to pulse splitting, or to some other phenomenon. This is formulated in the two open problems below.

\begin{openproblem}
Prove that the solutions to equation \eqref{rg13} are global in dimension 2 and 3.
\end{openproblem}

\begin{openproblem}
Describe rigorously the phenomenon of pulse splitting for equation \eqref{rg13}.
\end{openproblem}

\begin{remark}
We only considered in this section the case where,
$$P_2=I, \alpha_2=0,  \boldsymbol{\alpha}_3=0\textrm{ and }f=0$$
which is \eqref{rg13}. While it is of interest to investigate the qualitative role of the parameters $P_2, \alpha_2,  \boldsymbol{\alpha}_3$ and $f$ of equation \eqref{NLSfamily1} in the case of normal GVD (i.e. the case $\alpha_1=-1$), we chose not to investigate their effects since we are interested in this paper primary on the non existence of focusing dynamics,  and equation \eqref{rg13} is expected to have only global solutions. 
\end{remark}

\begin{remark}
In a recent preprint \cite{KNZ}, standing waves and selfsimilar solutions have been exhibited for equation \eqref{rg13} in dimension 2. However, these solutions do not belong to the energy space. In particular, they are not in $C^2$ and do not decay along the two diagonals in $\R^2$. 
\end{remark}

\subsection{Mixing several phenomena}\label{sec:rugby3}

In section \ref{sec:rugby1} and section \ref{sec:rugby2}, we have investigated the effect of each of the parameters  $P_2, \alpha_1, \alpha_2,  \boldsymbol{\alpha}_3$ and $f$ on the solutions to equation \eqref{NLSfamily1}, and in particular whether these modifications of the cubic NLS \eqref{eqNLS} prevent the existence of finite time blow-up solutions. We have also formulated several open problems. Now, instead of studying all these effects separately, one may investigate the case when all these phenomenon are present at the same time, and ask which are the dominant ones? Consider for instance the case where there is partial off-axis variation of the group velocity in the direction $z$, and at the same time a potential self-steepening of the pulse in the same direction, i.e. $P_2=1-\partial^2_z$ and $ \boldsymbol{\alpha}_3=e_z$,
\begin{equation}\label{rg16}
i (1-\varepsilon^2\partial^2_z)\dt v + \Delta v +(1+i \varepsilon\partial_z)(\abs{v}^2 v)=0.
\end{equation}
Then, one may wonder whether the partial off-axis variation of the group velocity prevents the self-steepening from taking place. This is formulated in the open problem below. 

\begin{openproblem}
Do the solutions to equation \eqref{rg16} exhibit the phenomenon of self-steepening?
\end{openproblem}

\subsection{The vectorial case}\label{sec:rugby4}

Recall that equation (\ref{NLSfamily1}) is a particular case of the vectorial equation (\ref{NLSfamily1})$_{\rm vect}$,
$$i P_2(\eps \nabla)\dt {\bf v} + (\Delta_\perp +\alpha_1\dz^2) {\bf v} +i \alpha_2 {\bf v}
 +\frac{1}{3}(1+i \eps  \boldsymbol{\alpha}_3\cdot \nabla )
\big[({\bf v}\cdot {\bf v})\overline{{\bf v}}+2\abs{{\bf v}}^2 {\bf v} \big]=0,
$$
where ${\bf v}$ is now a $\C^2$-valued function, and where we consider for simplicity the equation corresponding to the cubic case (i.e. $f=0$ in (\ref{NLSfamily1})). In fact, (\ref{NLSfamily1}) is a particular case of (\ref{NLSfamily1})$_{\rm vect}$ corresponding to initial data living on a one-dimensional subspace of $\C^2$ (see Remark \ref{remvect}). 
Now, one may of course consider the previous questions formulated for equation (\ref{NLSfamily1}), and investigate the same problems for the vectorial case. This is formulated in the open problem below.

\begin{openproblem}
Investigate the vectorial counterparts for equation (\ref{NLSfamily1})$_{\rm vect}$ of the various open problems formulated for the scalar equation (\ref{NLSfamily1}) in section \ref{sec:rugby1}, section \ref{sec:rugby2}, section \ref{sec:rugby3} and section \ref{sec:rugby5}. 
\end{openproblem}

\subsection{The approximation of the Maxwell equations over longer times}\label{sec:rugby5}

We have provided in \S \ref{sectDer} a rigorous justification for all the NLS-type models derived in this paper. 
More precisely, we have proved that there exists $T>0$ such that
\begin{enumerate}
\item The exact solution $\bU$ of (\ref{eqgen})-(\ref{CI}) exists on $[0,T/\eps]$.
\item The NLS-type model under consideration (e.g. (\ref{eqpolar})) admits a unique solution $u_{\rm app}$ with initial condition $u^0$  on the same time interval.
\item The approximation ${\bf U}_{\rm app}(t,x)=
u_{\rm app}(t,x)e^{i\frac{\bk\cdot x-\omega t}{\eps}}+\cc$ remains close to 
${\bf U}$ on this time interval,\footnote{In the error estimates of 
\S \ref{sectDer}, the constant on the r.h.s. is 
$C(T,\abs{u^0}_B)$ rather than 
$C(T,\abs{u_{\rm app}}_{L^\infty([0,T/\eps];B)})$. 
Since $\abs{u_{\rm app}}_{L^\infty([0,T/\eps];B)}
=C(T,\abs{u_0}_B)$, 
this is of course equivalent, but the first form is more
convenient for the present discussion. }
$$
\abs{\bU-\bU_{\rm app}}_{L^\infty([0,T/\eps]\times\R^d)}
\leq \eps C(T,\abs{u_{\rm app}}_{L^\infty([0,T/\eps];B)})
(1+\abs{\nabla u^0}_B+\abs{u^0}_{B^{(3)}}).
$$
\end{enumerate}
Such a justification is far from being sharp. In particular, for the standard cubic focusing NLS equation, one
has $T<T_{\rm foc}$, where $T_{\rm foc}$ is the time when focusing occurs. Indeed, the above justification process requires
that $u_{\rm app}$ remains bounded on $[0,T/\eps]$ (after rescaling, this is equivalent to require that the solution $v$ to (\ref{NLSfamily1}) is bounded on $[0,T]$). In the error estimate above, the constant $C(T,\abs{u_{\rm app}}_{L^\infty([0,T/\eps];B)})$ therefore
blows up as $T\to T_{\rm foc}$. \\
It follows that all the NLS-type equations derived here are justified far enough from the focusing time. Now, relevant differences between the different models can only be observed close enough to focusing (they all differ by formally
$O(\eps^2)$ terms that become relevant only near the focusing point). So, roughly speaking, we have only proved that all \emph{the models are justified on a time interval for which they are basically identical}.

As seen above, existence beyond the focusing time of the standard cubic NLS equation is proved or expected for
many of the variants considered here. For such models, the above argument does not work, i.e., the constant
$C(T,\abs{u_{\rm app}}_{L^\infty([0,T/\eps];B)})$ of the error estimate does not blow up as $T\to T_{\rm foc}$. This is for instance the case of
the cubic/quintic NLS equation (\ref{rg2}) that admits a global solution $v^\eps$. However (see Open Problem \ref{OP2}), we have
$$
\lim_{\eps\to 0}\lim_{T\to T_{\rm foc}} \abs{v^\eps(t,\cdot)}_{H^1}=+\infty,
$$
so that there is no reason to expect the error term $\eps C(T,\abs{v^\eps}_{L^\infty([0,T/\eps];B)})$ to be small near the focusing point and for small $\eps$.

For the moment, the merits of the NLS-type models derived here to describe correctly the mechanisms at stake during focusing can only by assessed numerically. Hence the following interesting open problem,
\begin{openproblem}
Rigorously justify one of the NLS-type models derived here on a time scale $[0,T/\eps]$ with $T\geq T_{\rm foc}$.
\end{openproblem}

\begin{acknowledgements}
The autors warmly thank Christof Sparber for pointing reference \cite{TVZ} out. 
\end{acknowledgements}

\appendix
\section{Nondimensionalization of the equations}

\subsection{The case without charge nor current density}\label{NDwithout}

There are two characteristic times for the situation considered here. The first one, denoted $\bar t$, is the inverse of the frequency of the laser pulse; the second one, denoted $\bar T$ is the duration of the pulse. In the regimes considered here\footnote{We refer to \cite{Donnat} for variants of this nondimensionalization.}, one
has $\bar T\gg \bar t$, and the small parameter $\eps$ is defined as
$$
\eps=\frac{\bar t}{\bar T}.
$$
The time variable is naturally nondimensionalized with $\bar T$, while the space variables are nondimensionalized by $L=c\bar T$, namely,
$$
t=\bar T \tilde t,\qquad x=L \tilde x,
$$
where dimensionless quantities are denoted with a tilde. 
We also nondimensionalize the unknowns $\ttE$, $\ttB$ and 
$\ttP$ with typical orders $\ttE_0$, $\ttB_0$ and $\ttP_0$. 
The typical scale for the electric and magnetic fields are 
directly deduced from $\ttP_0$,
$$
\ttE_0=\frac{1}{\epsilon_0}\ttP_0,\qquad 
\ttB_0=\frac{1}{c}\ttE_0.
$$

In order to choose the remaining $\ttP_0$, some consideration 
on the polarization equation is helpful. We recall that 
equation (\ref{eqpolarisation}) is given by
$$
\dt^2 \ttP+\Omega_1\dt \ttP+\Omega_0^2 \ttP-\nabla V_{NL}(\ttP)=\epsilon_0 b \ttE.
$$
For all the applications we have in mind, on has $\nabla V_{NL}(\ttP)=a_3\abs{\ttP}^2\ttP+h.o.t.$, so that we  nondimensionlize the 
nonlinear term as
$$
\nabla_\ttP V_{NL}(\ttP)=a_3 \ttP_0^3\nabla_{\tilde \ttP}\tilde V_{NL}(\tilde \ttP), 
$$
where $\tilde{V}_{NL}$ is the dimensionless potential
$$
\tilde{V}_{NL}(\tilde \ttP)=\frac{1}{a_3 \ttP_0^4}V_{NL}(\ttP).
$$
\begin{example}\label{exNLbis}
For the three cases considered in Example \ref{exNL}, we find, denoting  $\tilde a \eps^r=\frac{a_5\ttP_0^2}{a_3}$, where
$r >0$ is chosen such that $\tilde a=O(1)$,
\begin{itemize}
\item[(i)] Cubic nonlinearity: 
$$
\nabla_{\tilde \ttP}\tilde V_{NL}(\tilde \ttP)=\abs{\tilde \ttP}^2 \tilde \ttP 
$$
\item[(ii)] Cubic/quintic  nonlinearity:
$$
\nabla_{\tilde \ttP}\tilde V_{NL}(\tilde \ttP)= (1-\tilde a\eps^r  \abs{\tilde \ttP}^2)\abs{\tilde \ttP}^2\tilde \ttP.
$$
\item[(iii)] Saturated nonlinearity:
$$
\nabla_{\tilde \ttP}\tilde V_{NL}(\tilde \ttP)= \frac{1+\frac{\tilde a\eps^r}{3} \abs{\tilde \ttP}^2}{(1+\frac{2\tilde a\eps^r}{3}  \abs{\tilde \ttP}^2)^2}\abs{\tilde \ttP}^2\tilde \ttP.
$$
\end{itemize}
All these examples can therefore be put under the form
$$
\nabla_{\tilde \ttP}\tilde V_{NL}(\tilde \ttP)=\big(1+\tilde f(\tilde a\eps^r\abs{\tilde \ttP}^2)\big)\abs{\tilde \ttP}^2\tilde \ttP,
$$
where $\tilde f:\R^+\to \R$ is a smooth function such that $\tilde f(0)=0$.
\end{example}
In dimensionless form, the polarization  equation becomes therefore
$$
\partial_{\tilde t}^2 \tilde \ttP+\bar T\Omega_1 \partial_{\tilde t}\tilde \ttP +\bar T^2\Omega_0^2 \tilde \ttP
-a_3\bar T^2\ttP_0^2 \nabla_{\tilde \ttP}\tilde V_{NL}(\tilde \ttP) =\epsilon_0 b\frac{\bar T^2 \ttE_0}{\ttP_0}\tilde \ttE
$$
We therefore choose $\ttP_0$ to be the typical nonlinear scale for which the nonlinear terms in the above equations are
of the same order as the second order time derivative, namely,
$$
\ttP_0=\frac{1}{\bar T\sqrt{a_3}}.
$$

Finally, we need some information on the size of $\Omega_0$, $\Omega_1$ and $b$ to give our final dimensionless form of the
equations. The resonance frequency of the harmonic oscillator 
is typically of the same order as the frequency of the laser 
pulse,
$$
\Omega_0=\bar t^{-1}\omega_0;
$$
the damping frequency is written under the form
$$
\Omega_1=\eps^{\alpha+1}\bar t^{-1} \,\omega_{1},
$$
where the dimensionless damping frequency $\omega_{1}$ has size $O(1)$ while the coefficient $\alpha$ depends on the medium in which the laser propagates. Typically, $\alpha=1$ or $\alpha=2$. Writing the coupling constant $b$ as
$$
b=\frac{\gamma}{\bar t^2},
$$
the dimensionless polarization equations can therefore be written as 
$$
\partial_{\tilde t}^2 \tilde \ttP+\eps^{\alpha}\omega_1 \partial_{\tilde t}\tilde \ttP +\eps^{-2}\omega_0^2 \tilde \ttP
- \nabla_{\tilde \ttP}\tilde V_{NL}(\tilde \ttP)=\eps^{-2}\gamma\tilde \ttE.
$$
Omitting the tildes for the sake of clarity, Maxwell's equations (\ref{Max2})-(\ref{eqpolarisation})  read therefore
(without  charge nor current density),
$$
\left\lbrace
\begin{array}{l}
\dsp \dt \ttB+ \curl \ttE=0,\vspace{1mm}\\
\dsp \dt \ttE- \curl \ttB=-\dt \ttP,\vspace{1mm}\\ 
\dsp \partial_{t}^2 \ttP+\eps^{\alpha}\omega_1 \partial_{t}\ttP +\eps^{-2}\omega_0^2  \ttP
-\nabla  V_{NL}( \ttP)=\eps^{-2}\gamma \ttE.
\end{array}\right.
$$
Introducing the auxiliary unknowns,
$$
\ttP^\sharp=\frac{\omega_0}{\sqrt{\gamma}}\ttP,\qquad
\ttQ^\sharp=\frac{\eps}{\omega_0}\dt \ttP^\sharp=\frac{\eps}{\sqrt{\gamma}}\dt \ttP,
$$
and working with nonlinearities as those considered in Example \ref{exNLbis}, 
this system can be written as a first order system,
\begin{equation}\label{MaxND}
\left\lbrace
\begin{array}{l}
\dsp \dt \ttB+ \curl \ttE=0,\vspace{1mm}\\
\dsp \dt \ttE- \curl \ttB+\frac{1}{\eps}\sqrt{\gamma} \ttQ^\sharp=0,\vspace{1mm}\\ 
\dsp \dt \ttQ^\sharp+\eps^{\alpha}\omega_1 \ttQ^\sharp-\frac{1}{\eps}(\sqrt{\gamma} \ttE-\omega_0 \ttP^\sharp)=\eps \frac{\gamma}{\omega_0^3}
\big(1+f(\eps^r\abs{\ttP^\sharp}^2)\big)\abs{ \ttP^\sharp}^2\ttP^\sharp,
\vspace{1mm}\\ 
\dsp \dt \ttP^\sharp-\frac{1}{\eps}\omega_0 \ttQ^\sharp=0,
\end{array}\right.
\end{equation}
where we wrote $f(x)=\tilde f(\tilde a \frac{\gamma}{\omega_0^2}x)$.\\
This system has the form (\ref{eqgen})  with
 $n=12$, $\bU=(\ttB,\ttE,\ttQ^\sharp,\ttP^\sharp)^T$ and
$$
A(\partial)=
\left(
\begin{array}{cccc}
0 & \curl & 0 & 0\\
-\curl & 0 & 0 & 0\\
0 & 0 & 0 &0\\
0 & 0 & 0 &0
\end{array}\right)
,\qquad
E=\left(
\begin{array}{cccc}
0 & 0 & 0 & 0\\
0 & 0 & \sqrt{\gamma}I & 0\\
0 & -\sqrt{\gamma}I & 0 & \omega_0 I \\
0 & 0 & -\omega_0 I & 0
\end{array}\right)
$$
while $A_0$ is the block diagonal matrix $A_0=\mbox{diag}(0,0,\omega_1 I,0)$, 
so that the first two points of Assumption \ref{assu1} are satisfied. For the third one, remark that the nonlinearity is given by
\begin{equation}\label{NLMax}
F(\eps,\bU)=\left(0,0,\frac{\gamma}{\omega_0^3}(1+f(\eps^r\abs{\ttP^\sharp}^2) \big)\abs{\ttP^\sharp}^2\ttP^\sharp,0\right)^T,
\end{equation}
which is of the same form as in Assumption \ref{assu1} 
with $Q(U)=\abs{\ttP^\sharp}^2$ and $T(U,\bar U,U)=(0,0,
\frac{\gamma}{\omega_0^3}\abs{\ttP^\sharp}^2\ttP^\sharp,0)^T$. 
One can check that such nonlinearities also satisfy 
Assumption \ref{assuF}.

\subsection{The case with charge and current density}
\label{NDwith} 
Now, we deal with the nondimensionalization of equations 
\eqref{Max2}-\eqref{eqpolarisation}-\eqref{eqrhoJ}. We proceed 
in the same way as in the previous paragraph, simply adding 
the characteristic sizes $\ttJ_0$ and $\rho_0$ of the free electron current 
density $\ttJ_{\rm e}$ and charge density $\rho$. Considerations on 
the polarization equation are the same. 

\noindent
We choose $\rho_0$ from the equation for the charge density,
$$
\dsp \partial_{\tilde t} \tilde\rho = \sigma_K \rho_{\rm nt} 
\frac{\ttE_0^{2K}}{\rho_0} \bar T \abs{\tilde\ttE}^{2K}
+ \frac{\sigma}{U_{\rm i}} \ttE_0^2 \bar T 
\tilde\rho \abs{\tilde\ttE}^2.
$$
Many configurations are possible, according to the numerical value of
the various 
coefficients involved (see for instance  \cite{BergeSkupin} for
experimental data). We choose here a configuration that leads 
to the richest model, where the coupling constants of the two 
nonlinear terms are of the same order; more precisely
$$
\dsp \partial_{\tilde t} \tilde\rho = 
\eps c_1 \abs{\tilde\ttE}^{2K}
+ \eps c_2 \tilde\rho \abs{\tilde\ttE}^2, 
$$
for non-negative constants $c_1$, $c_2$. This corresponds to the
following choice for $\rho_0$,
$$
\rho_0=\frac{\sigma_K\rho_{\rm nt}E_0^{2K}}{\eps c_1}.
$$
Knowing $\rho_0$, we can then determine $\ttJ_0$ from the equation for the
free electron current density,
$$
\ttJ_{\rm e}=\frac{q_{\rm e}^2}{\omega_{\rm l} m_{\rm e}}\rho_0 E_0\,\tilde \rho\,{\mathcal H}(\eps D_{\tilde z})\tilde \ttE,
$$
which naturally leads to
$$
\ttJ_0=\frac{q_{\rm e}^2}{\omega_{\rm l} m_{\rm e}}\rho_0 E_0=\eps \frac{q_{\rm e}^2\bar{T}\rho_0 E_0}{m_{\rm e}}.
$$
Going back to Maxwell's 
equations, we get 
$$
\left\lbrace
\begin{array}{l}
\dsp \partial_{\tilde t} \tilde\ttB 
+ \curl \tilde\ttE = 0 , \vspace{1mm}\\
\dsp \partial_{\tilde t} \tilde\ttE - \curl \tilde\ttB
= - \partial_{\tilde t} \tilde\ttP 
- \frac{1}{\epsilon_0} \bar T \frac{\ttJ_0}{\ttE_0} 
\tilde\ttJ_{\rm e} -\frac{U_{\rm i}\rho_0}{\epsilon_0 E_0}\big(\eps c_1
\abs{\tilde\ttE}^{2K-2}+\eps c_2\tilde\rho\big)\tilde\ttE , \vspace{1mm} \\ 
\end{array}\right.
$$
Here again, many configurations can be found, and we choose one that
leads to the richest models, namely,
$$
\frac{1}{\epsilon_0} \bar T \frac{\ttJ_0}{\ttE_0} =\eps c_3,\qquad
\frac{U_{\rm i}\rho_0}{\epsilon_0 E_0}=c_0
$$
where $c_0,c_3$ are dimensionless positive constants. Replacing $\tilde\ttJ_{\rm e}$ by
$c_3\tilde\ttJ_{\rm e}$, $\tilde\rho$ by $c_3\tilde\rho$ and $c_1$ by $c_1c_3$, we
can assume that $c_3=1$.

Finally, dropping the tildes, introducing the unknowns 
$\ttQ^\sharp$ and $\ttP^\sharp$, and using the same notations 
as in \eqref{MaxND}, we obtain
\begin{equation}\label{MaxNDwith}
\left\lbrace
\begin{array}{l}
\dsp \dt \ttB + \curl \ttE = 0 , \vspace{1mm} \\
\dsp \dt \ttE - \curl \ttB + \frac{1}{\eps}\sqrt{\gamma} \ttQ^\sharp =
-\eps \rho {\mathcal H}(\eps D_z)\ttE 
-\eps c_0 \big(c_1\abs{\ttE}^{2K-2}+ c_2\rho\big)\ttE , \vspace{1mm} \\ 
\dsp \dt \ttQ^\sharp 
+ \eps^{\alpha}\omega_1 \ttQ^\sharp 
- \frac{1}{\eps}(\sqrt{\gamma} \ttE 
- \omega_0 \ttP^\sharp)=\eps \frac{\gamma}{\omega_0^3}
\big(1+f(\eps^r\abs{\ttP^\sharp}^2)\big)
\abs{ \ttP^\sharp}^2\ttP^\sharp , \vspace{1mm} \\ 
\dsp \dt \ttP^\sharp 
- \frac{1}{\eps}\omega_0 \ttQ^\sharp=0 , \vspace{1mm} \\ 
\dsp \dt \rho = 
\eps c_1 \abs{\ttE}^{2K} + \eps c_2 \rho \abs{\ttE}^2.
\end{array}\right.
\end{equation}

\section{Explicit computations for Maxwell equations}
\label{sectexplcomp}

\subsection{The case without charge nor current density}

We derive here the variants of the NLS equations derived in this paper in the particular case of the 
Maxwell equations (\ref{MaxND}). We derive all the versions of the NLS equations that do not take
into account the frequency dependence of the polarization (this corresponds to (\ref{eqrad2})). This latter effect,
which leads to the generalization (\ref{eqrad3}) of the previous model, is examined separately. For the sake of simplicity, we only consider cubic nonlinearities here; we show that in this context, (\ref{eqrad3}) reduces to the
vectorial family of NLS equation $\eqref{NLSfamily1}_{\rm vect}$.
\subsubsection{Without frequency dependent polarization}
In order to check that Assumption \ref{assu3} is  satisfied by the dimensionless Maxwell equations (\ref{MaxND}),
we need to compute the eigenvalues and eigenvectors of the matrix
$$
{\mathcal L}(0,\bk)=A(\bk)+\frac{1}{i}E,
$$
where $A(\bk)$ and $E$ are given in \S \ref{NDwith}, which is equivalent to find all non trivial solutions to the equation
$$
{\mathcal L}(\omega,\bk)u=0, \quad \mbox{ with }\quad {\mathcal L}(\omega,\bk)=-\omega I+A(\bk)+\frac{1}{i}E,
$$
and $\omega\in \R$, $u=(\ttb,\tte,\ttq^\sharp,\ttp^\sharp)\in \C^{12}$.\\
From the last two equations of (\ref{MaxND}), we can always write $\ttp^\sharp$ 
and $\ttq^\sharp$ in terms of $\tte$,
\begin{equation}\label{relpor}
\ttp^\sharp=-\frac{\omega_0\sqrt{\gamma}}{\omega^2-\omega_0^2}\tte,\qquad
\ttq^\sharp=i\frac{\omega\sqrt{\gamma}}{\omega^2-\omega_0^2}\tte.
\end{equation}
Let us now consider several cases:
\begin{itemize}
\item If $\tte\parallel \bk$ then by the first equation, either $\omega=0$ and $\ttb \parallel \bk$ or $\ttb=0$ and $\omega^2=\omega_0^2+\gamma$.
\item If $\tte\not\parallel\bk$ then $\omega\neq 0$ and combining the first two equations, we get
$$
\big[\omega^2\frac{\omega^2-(\omega_0^2+\gamma)}{\omega^2-\omega_0^2}-k^2\big]\tte=-(\bk\cdot \tte)\bk,
$$
where we denoted as usual by $k$ the modulus of $\bk$, $k=\abs{\bk}$; this
 in turns implies (since $\bk$ is not parallel to $\tte$) that $\bk\perp\tte$ and
$$
\omega^4-\omega^2\big((\omega_0^2+\gamma)+k^2\big)+\omega_0^2 k^2=0;
$$
denoting $\Delta=\big((\omega_0^2+\gamma)+k^2\big)^2-4\omega_0^2k^2>0$, this equation admits four real solutions
$$
\omega_{\pm,\pm}(k)=\pm\frac{1}{\sqrt{2}}\sqrt{\omega_0^2+\gamma+k^2\pm \sqrt{\Delta}}.
$$
\end{itemize}

There are $m=7$ different branches of $\CL$: 
three constant sheets given by
$$
\omega=0 \quad \text{and} \quad 
\omega=\pm\sqrt{\gamma+\omega_0^2}
$$
(the first one being of multiplicity two, the two other ones of multiplicity one), 
and four curved sheets (each one being of multiplicity two)
given by $\omega=\omega_{\pm,\pm}$.  
We denote by $\pi_{\pm,\pm}$ the associated eigenprojectors 
(which are therefore of rank 2). 
Remarking that $\pi_{\pm,\pm}(\bk)u=u$ if and only if
$$
\tte\cdot \bk=0,\qquad 
\ttb=-\frac{1}{\omega} \bk\wedge \tte, 
$$
and $\ttp^\sharp$ and $\ttq^\sharp$ are given in terms of $\tte$ by (\ref{relpor}), one readily finds that  
$\pi_{\pm,\pm}(\bk)$ can be written as a block matrix 
$$
\pi_{\pm,\pm}(\bk)=\frac{1}{N(k)^2}(P_{ij})_{1\leq i,j\leq 4}
\quad\mbox{ with }\quad
N(k)^2=\frac{k^2}{\omega(k)^2}+1+\gamma\frac{\omega(k)^2+\omega_0^2}{(\omega(k)^2-\omega_0^2)^2}
$$
(for the sake of clarity, we write $\omega(k)=\omega_{\pm,\pm}(k)$), and
 where the $3\times3$ blocks
$P_{1j}$ are given by
\begin{eqnarray*}
P_{11}=\frac{k^2}{\omega(k)}\Pi_{\bk^\perp},
& &
P_{12}=-\frac{1}{\omega(k)}\bk\wedge,\\
P_{13}=i\frac{\sqrt{\gamma}}{\omega(k)^2-\omega_0^2}\bk\wedge,
& &
P_{14}=\frac{\omega_0}{\omega(k)}\frac{\sqrt{\gamma}}{\omega(k)^2-\omega_0^2}\bk\wedge,\\
P_{22}=\Pi_{\bk^\perp},
& &
P_{23}=-i\frac{\omega(k)\sqrt{\gamma}}{\omega(k)^2-\omega_0^2}\Pi_{\bk^\perp},\\
P_{24}=-\frac{\omega_0\sqrt{\gamma}}{\omega(k)^2-\omega_0^2}\Pi_{\bk^\perp},
& &
P_{33}=\frac{\omega(k)^2\gamma}{(\omega(k)^2-\omega_0^2)^2}\Pi_{\bk^\perp},\\
P_{34}=-i\frac{\omega(k)\omega_0\gamma}{(\omega(k)^2-\omega_0^2)^2}\Pi_{\bk^\perp},
& &
P_{44}=\frac{\omega_0^2\gamma}{(\omega(k)^2-\omega_0^2)^2}\Pi_{\bk^\perp},
\end{eqnarray*}
where $\Pi_{\bk^\perp}$ is the orthogonal projector onto the orthogonal plane to $\bk$; the remaining blocks stem from
the symmetry relations $P_{ij}=P_{ji}^*$.

It follows from these computations that in the case of a cubic nonlinearity\footnote{The treatment of other
kinds of nonlinearities is absolutely similar} (i.e. $f=0$ in (\ref{NLMax})), $F^{env}(\eps,u)$ is given for all $u$ such that
$u=\pi_{\pm,\pm}(\bk)u$ by
$$
F^{env}(\eps,u)=\big(0,0,-\frac{\gamma^{5/2}}{(\omega^2-\omega_0^2)^3}[(\tte\cdot \tte)\bar \tte+2\abs{\tte}^2\tte],0\big)^T;
$$
similarly, 
$$
\pi_{\pm,\pm}(\bk)F^{env}(\eps,u)=\big(*,\frac{i}{N(k)^2}\frac{\gamma^{3}\omega}{(\omega^2-\omega_0^2)^4}[(\tte\cdot \tte)\bar \tte+2\abs{\tte}^2\tte],*,*\big)^T.
$$
and
$$
\pi_{\pm,\pm}A_0u=\big(*,\omega_1\frac{\gamma\omega^2}{(\omega^2-\omega_0^2)^2}\tte,*,*\big)^T.
$$

We can now write explicitly the general nonlinear Schr\"odinger equation (\ref{eqrad2})
derived in Example \ref{exrad2} in the particular case
of the Maxwell equations (\ref{MaxND}). Choosing $\omega_1(\cdot)=\omega_{\pm,\pm}(\cdot)$ and therefore
$\pi_1(\cdot)=\pi_{\pm,\pm}(\cdot)=\omega(\cdot)$, and remarking that the solution $v$ to (\ref{eqrad2}) 
remains polarized along $\pi_{\pm,\pm}(\bk)$ if it initially polarized, it is enough to give an equation on its electric field component ${\mathtt e}$
(the components $\ttb$, $\ttp^\sharp$ and $\ttq^\sharp$ being recovered as indicated above). Moreover, we can 
assume the $\bk=k{\bf e}_z$, so that the electric field $\tte$ only has transverse components $\fre\in \C^2$ (i.e. $\tte=(\fre^T,0)^T$) and (\ref{eqrad2}) finally reduces to
\begin{eqnarray}
\nonumber
\lefteqn{(1-i\eps {\bf b}\cdot \nabla -\eps^2\nabla\cdot  B\nabla)\partial_t \fre-\frac{i}{2}\big(\frac{\omega'(k)}{k}\Delta_\perp +\omega''(k)\dz^2\big) \fre}\\
\label{marredemarre}
& &+\eps^p\omega_1\frac{\gamma\omega(k)^2}{(\omega(k)^2-\omega_0^2)^2}\fre
=
 \frac{i}{N(k)^2}\frac{\gamma^{3}\omega(k)}{(\omega(k)^2-\omega_0^2)^4}[(\fre\cdot \fre)\bar \fre+2\abs{\fre}^2\fre].
\end{eqnarray}
Proceeding to the following rescaling of the space variables
$$
(x,y)=\big(\frac{\omega'(k)}{2k}\big)^{1/2}(x',y')
\quad\mbox{ and }\quad
z=\abs{\frac{1}{2}\omega''(k)}^{1/2}z'
$$
(the latter one only if $\omega''(k)\neq 0$), and
 rescaling $\fre$ as 
$$
{\mathfrak e}=\Big(\frac{1}{N(k)^2}\frac{\gamma^{3}\omega(k)}{3(\omega(k)^2-\omega_0^2)^4}\Big)^{1/2}\fre',
$$
the above equation becomes
\begin{equation}\label{AB}
iP_2(\eps\nabla)\dt\fre+\big(\Delta_\perp +\alpha_1\dz^2\big) \fre +i\alpha_2\fre
+\frac{1}{3}[(\fre\cdot \fre)\bar \fre+2\abs{\fre}^2\fre]=0
\end{equation}
where $\alpha_1=\mbox{sgn}(\omega''_{\pm,\pm}(k))\in\{0,\pm 1\}$, $\alpha_2=\eps^p\omega_1\frac{\gamma\omega(k)^2}{(\omega(k)^2-\omega_0^2)^2}$ and
$$
P_2(\eps\nabla)=1-i\eps M{\bf b}\cdot\nabla-\eps^2\nabla \cdot MBM\nabla,
$$
where 
$M=\mbox{diag}\Big(\big(\frac{\omega'(k)}{2k}\big)^{-1/2},\big(\frac{\omega'(k)}{2k}\big)^{-1/2},\abs{\frac{1}{2}\omega''(k)}^{-1/2}\Big)$. This equation is exactly under the form $\eqref{NLSfamily1}_{\rm vect}$ with $\boldsymbol{\alpha}_3=0$ (i.e. without the terms modeling the frequency dependence of the polarization).

\subsubsection{With frequency dependent polarization}

Let us now describe the modifications that have to be made in order to take into account a frequency dependent 
polarization, as in (\ref{eqrad3}). Since initial data polarized along $\pi_{\pm,\pm}(\bk)$ conserve this polarization
during the evolution in time, the only difference between (\ref{eqrad3}) and (\ref{eqrad2}) is the presence of
the terms $\eps\pi_{\pm,\pm}(\bk)\pi_{\pm,\pm}'(\bk)\cdot D F^{env}(\eps,u)$ and $-i\eps({\bf b}\cdot\nabla)\pi_{\pm,\pm}(\bf k)F^{env}(\eps,u)$. Only the first one requires a nontrivial computation. 
Thanks to the expression for $\pi_{\pm,\pm}(\bk)$ given  above, we can write for all $F=(0,0,f,0)$, 
$$
\pi_{\pm,\pm}(\bk)F=(ig_1(k)\bk\wedge f,-ig_2(k)\Pi_{\bk^\perp}f,g_3(k)\Pi_{\bk^\perp}f,ig_4(k)\Pi_{\bk^\perp}f)^T,
$$
with 
\begin{eqnarray*}
g_1(k)=\frac{1}{N(k)^2}\frac{\sqrt{\gamma}}{\omega(k)^2-\omega_0^2},& &
 g_2(k)=\omega(k)g_1(k),\\
g_3(k)=\frac{\omega(k)\sqrt{\gamma}}{\omega(k)^2-\omega_0^2}g_2(k), & &
g_4(k)=\frac{\omega_0\sqrt{\gamma}}{\omega(k)^2-\omega_0^2}g_2(k).
\end{eqnarray*}
One has therefore, with $h=\frac{\bk}{k} \cdot D f$ and $l=d_{\bk}(\bk'\mapsto \Pi_{\bk'\,^\perp})\cdot D f$,
\begin{eqnarray*}
\pi_{\pm,\pm}'(\bk)\cdot DF&=&(ig_1'(k)\bk\wedge h,-ig_2'(k)\Pi_{\bk^\perp} h,g_3'(k)\Pi_{\bk^\perp}h,ig_4'(k)\Pi_{\bk^\perp}h)^T\\
&+&(ig_1(k) D\wedge f,-ig_2(k)l,g_3(k)l,ig_4(k)l)^T . 
\end{eqnarray*}
If $f\cdot \bk=0$, we can use the identity $\Pi_{\bk^\perp}d_{\bk}(\bk'\mapsto \Pi_{\bk'\,^\perp})\cdot D\Pi_{\bk^\perp}=0$
to deduce that
$$
\pi_{\pm,\pm}(\bk)\pi_{\pm,\pm}'(\bk)\cdot DF=\big(*,-\frac{m(k)}{N(k)^2}\frac{\bk}{k}\cdot \nabla f,*,*\big)^T,
$$
where the $*$ can be easily computed, and where
$$
m(k)=\frac{k}{\omega(k)}(g_1(k)+kg_1'(k))+g_2'(k)+\frac{\sqrt{\gamma}}{\omega(k)^2-\omega_0^2}(\omega(k)g_3'(k)-\omega_0 g_4'(k)).
$$
The nonlinearity in the r.h.s. of (\ref{marredemarre}) must therefore be replaced by
$$
 \frac{i}{N(k)^2}\frac{\gamma^{3}\omega(k)}{(\omega(k)^2-\omega_0^2)^4}(1-i\eps {\bf b}\cdot\nabla+i\alpha_3\dz)[(\fre\cdot \fre)\bar \fre+2\abs{\fre}^2\fre],
$$
with $\alpha_3=-\frac{\omega^2-\omega_0^2}{\sqrt{\gamma}\omega(k)}$. The reduction to $\eqref{NLSfamily1}_{\rm vect}$ then follows the same steps as for (\ref{AB}).

\subsection{The case with charge and current density} \label{APPmg}

According to \eqref{NLSpbis} and more generally \eqref{eqpolarioniz},
ionizations effects are taken into account by adding
$$
I:=-\eps \pi_1(\bk) \left( i w C_1^T C_1 u 
+ c \, C_1^T G^{env}(C_1 u,w) \right) 
$$
to the right-hand side of the NLS equation, and by considering the
following equation for $w$,
$$
\dsp \dt w = 2 \eps G^{env}(C_1 u,w) \cdot \overline{C_1 u} .
$$
In the particular case of the Maxwell equations \eqref{Max4}, one has
$w=\rho$ and
$$
I=-\eps \pi_1({\bf k})\Big(0,i\rho \tte+c_0\big(c_1\abs{\tte}^{2K-2}+c_2\rho\big)\tte,0,0\Big)^T;
$$
consequently, one must add to the right-hand side of
\eqref{marredemarre} the ionization term
$$
-i\rho {\fre}-c_0\big(c_1\abs{\fre}^{2K-2}+c_2\rho\big)\fre , 
$$
so that after rescaling as for \eqref{AB}, the equation takes the form
\eqref{rg4} (when written in a fixed frame rather than the frame
moving at the group velocity).

\end{document}